\documentclass{amsart}
\usepackage{amssymb,latexsym,amsmath,extarrows,amsthm, tikz}
\usepackage{hyperref}
\usepackage{graphicx}
\usepackage{mathabx}
\usepackage{enumerate}
\usepackage{ulem} %allows \sout and \xout, note that it also adds underling to the bib (unless you add the code \normalulem before the bib, and it changes emph to underline throughout

\DeclareMathAlphabet{\pazocal}{OMS}{zplm}{m}{n} %allows pazocal lettering
\usepackage{color}
\definecolor{blue}{rgb}{0,0,1}

\definecolor{red}{rgb}{1,0,.2}

\theoremstyle{plain}
\newtheorem{thm}{Theorem}[section]
\newtheorem{lem}[thm]{Lemma}
\newtheorem{prop}[thm]{Proposition}
\newtheorem{defn}[thm]{Definition}

\newtheorem{cor}[thm]{Corollary}

\newtheorem{example}[thm]{Example}
\numberwithin{equation}{section}

\newcommand{\R}{\ensuremath{\mathbb{R}}}

\DeclareMathOperator{\Fav}{Fav}

\DeclareMathOperator{\FavG}{\Fav_{\Gamma}}
\DeclareMathOperator{\K}{\mathcal{K}}

\newcommand{\dimh}{\operatorname{dim}_{\mathcal{H}}}

\newcommand{\proj}{\widetilde{\pi_{\alpha}}} 
\newcommand{\vis}{\operatorname{vis}}

\newcommand{\al}{\alpha}
\newcommand{\be}{\beta}
\newcommand{\ga}{\gamma}
\newcommand{\de}{\delta}

\newcommand{\Ga}{\Gamma}

\newcommand{\la}{\lambda}
\newcommand{\La}{\Lambda}

\newcommand{\Om}{\Omega}

\newcommand{\set}[1]{\left\{#1\right\}}

\newcommand{\pr}[1]{\left( #1 \right) }

\usepackage{hyperref}

\author[Tyler Bongers]{Tyler Bongers}
\address{Tyler Bongers, Department of Mathematics, Harvard University}
\email{bongers@math.harvard.edu}

\author[Krystal Taylor]{Krystal Taylor}
\address{Krystal Taylor, Department of Mathematics, The Ohio State University}
\email{taylor.2952@osu.edu}
\thanks{Taylor is supported in part by the Simons Foundation Grant 523555.}

\title{Transversal families of nonlinear projections and generalizations of Favard length}
\begin{document}
\date{\today}

\begin{abstract}
Projections detect information about the size, geometric arrangement, and dimension of sets. To approach this, one can study the energies of measures supported on a set and the energies for the corresponding pushforward measures on the projection side. For orthogonal projections, quantitative estimates rely on a separation condition: most points are well-differentiated by most projections. It turns out that this idea also applies to a broad class of nonlinear projection-type operators satisfying a \textit{transversality condition}. In this work, we establish that several important classes of nonlinear projections are transversal. This leads to quantitative lower bounds for decay rates for nonlinear variants of Favard length, including Favard curve length (as well as a new generalization to higher dimensions, called Favard surface length) and visibility measurements associated to radial projections. As one application, we provide a simplified proof for the decay rate of the Favard curve length of generations of the four corner Cantor set, first established by Cladek, Davey, and Taylor.
\end{abstract}

\maketitle

%\date{\today}
\hypersetup{linktocpage}
  \setcounter{tocdepth}{1}
\tableofcontents
\section{Introduction and Main Results} 
The Favard length of a planar set $E$ is the average length of its orthogonal projections.
It is defined by 
$$\Fav(E) = \frac{1}{\pi}\int_0^\pi |P_\theta(E)| d\theta,$$
where $P_{\theta}$ is orthogonal projection into a line $L_{\theta}$ through the origin at angle $\theta$ from the positive $x$-axis and $|\cdot|$ denotes the $1$-dimensional Hausdorff measure.
 Favard length gives a $1$-dimensional notion of the \textit{size} of a set which takes into account the geometry, arrangement, and rectifiability of the underlying set. As a consequence, there are deep relationships between Favard length and analytic capacity, the understanding of which is related to important open problems in geometric measure theory. As we will see, variants of the Favard length can also be formulated for more general families of mappings, beyond the orthogonal projections, and in higher dimensions.
 
As the Hausdorff dimension of a set cannot increase under a projection, sets of dimension $s < 1$ have Favard length equal to zero. A refinement due to Marstrand \cite{Mar54} actually shows that the dimension of such a set will be preserved in almost every direction. On the other hand, sets with dimension $s > 1$ will have positive-length projections in almost every direction, and therefore have positive Favard length. 
Therefore, the critical dimension is $s = 1$.  

In dimension $1$, the key geometric property that Favard length can detect is
\textit{rectifiability}: it is a consequence of the Besicovitch projection theorem  \cite{Bes39} that 
purely unrectifiable sets in the plane with finite $1$-dimensional Hausdorff measure have Favard length equal to zero.  
(For an exposition of the full Besicovitch-Federer projection theorem in all dimensions, see \cite[Chapter 18]{Mat95}.) While Besicovitch's theorem gives a qualitative result, we can find related quantitative theorems. If $E(r)$ is the $r$-neighborhood of a set $E$ with Favard length zero, the dominated convergence theorem shows that
$$\lim_{r \to 0^+} \Fav(E(r)) = 0.$$

 More precise asymptotic information for $\Fav(E(r))$ as $r$ decreases to zero can give quantitative measurements of the dimension, size, and geometric arrangement of $E$. A number of authors have investigated quantitative versions of the Besicovitch projection theorem for general sets. The best known results in terms of upper and lower bounds are due to Tao \cite{Tao09} and Mattila \cite{Mat90} respectively.

Tao introduced a quantitative version of rectifiability for sets in the plane of finite $\mathcal{H}^1$ measure and used multiscale analysis to show that an upper bound on the so-called rectifiability constant yields an upper bound on the Favard length. A nonlinear version of Tao's theorem is studied in a work of Davey and the second listed author \cite{DT21}.

Mattila \cite{Mat90} established a fundamental relationship between the Favard length of a set and its Hausdorff dimension.  In two dimensions, it states:
\smallskip

\begin{thm}[Favard lengths for neighborhoods; Mattila \cite{Mat90}]\label{theorem:mattila90_main}
Fix $s \in (0, 1]$. If $F \subseteq \mathbb{R}^2$ is the support of a Borel probability measure with $\mu(B(x, r)) \le b r^s$ for all $x \in \mathbb{R}^2$ and $0 < r < \infty$, then 
$$\Fav(F(r)) \gtrsim r^{1 - s}$$
if $s < 1$ and
$$\Fav(F(r)) \gtrsim \left(\log r^{-1}\right)^{-1}$$
if $s = 1$.
\end{thm}
Throughout the paper, we will use the notation $A \lesssim B$ to mean that there is a constant $C$ so that $A \le C B$, and will write $A \sim B$ if $A \lesssim B$ and $B \lesssim A$. 

The proof of Mattila's result follows from studying \textit{energies}: if $\mu$ is a measure, its \emph{$s$-energy} is
\begin{equation}\label{energy}
I_s(\mu) = \iint \frac{d\mu(x) \, d\mu(y)}{|x - y|^s}.
\end{equation}
This quantity is closely tied to Hausdorff dimension; see, e.g. \cite[Chapter 8]{Mat95} for a formulation of the definition of Hausdorff dimension in terms of $s$-energies. In order to relate a measure to the projections, we need the notion of a pushforward: if $f : X \to Y$ is a function and $\mu$ is a measure supported on $X$ we will define the \emph{pushforward measure} $f_{\sharp} \mu$ by
\begin{equation}\label{pushforward} 
(f_{\sharp} \mu)(A) = \mu(f^{-1}(A)), \quad\quad A \subseteq Y. 
\end{equation}
In general it can be difficult to study the pushforward of under a particular mapping, yet it turns out that the \textit{average} energy of a projection can be well controlled.  
That is, if $\{\pi_{\alpha} : \alpha \in A\}$ is an indexed family of orthogonal projections, it frequently is possible to precisely estimate
$$\int I_t(\pi_{\alpha\sharp} \mu) \, d\psi(\alpha),$$
where $\psi$ is a measure on the index set $A$.  By studying the average energy of the pushforwards of specialized measures supported on $F(r)$ with particular density properties, Mattila was able to establish the stated lower bounds. Further details are given in Section \ref{energy_section}.

In the special setting that the underlying set is a fractal generated by an iterated function system, Mattila's techniques with energies are also applicable. A standard example of this is to consider the generations $K_n$ of the \emph{four corner Cantor set}; it is defined by dividing the unit square into $16$ axis parallel squares of side length $\frac{1}{4}$, keeping the four corner squares, and iterating the process within each corner. The limit of this process gives a prototypical example of a purely unrectifiable set with positive and finite length. As such, an important open problem is to estimate upper and lower bounds on the rate of decay in $n$ of $\Fav(\K_n)$ (see \cite{Laba12} for a survey of results and techniques related to this problem). 
Mattila's techniques can be used to show that $\Fav(\K_n) \gtrsim n^{-1}$; subsequent work has achieved the tighter bounds
\begin{equation}\label{K_n}
\frac{\log n}{n} \lesssim \Fav(\K_n) \lesssim \frac{1}{n^{1/6 - \delta}}
\end{equation}
for any $\delta > 0$, with the bounds due to Bateman and Volberg \cite{BaV10} and Nazarov, Peres, and Volberg \cite{NPV10} respectively. Further, it is still a deep open question whether the Favard length $\Fav(\K_n)$ is larger or smaller than the analytic capacity $\gamma(\K_n)$, which is known to be of order $n^{-1/2}$ \cite{Tolsa}.

The primary aim of this paper is to formulate Theorem \ref{theorem:mattila90_main} in a nonlinear setting for families of projections which are not orthogonal projections. In particular, we will consider families of maps satisfying the so-called transversality condition. After we establish a correspondence between the energy of a measure and its pushforwards under transversal families, we will apply these relationships to study the asymptotic decay rates of nonlinear variants of Favard length. In the process, we generalize the lower bounds on visibility established by Bond, {\L}aba, and Zahl \cite{BLZ} as well as provide a simplified proof of the lower bound for the Favard curve length of $\K_n$ derived by Cladek, Davey, and Taylor \cite{CDT21}; both of these results are explored in Section \ref{section_nonlinear_families}. Before stating our main results in Section \ref{section_main}, we give several examples of families of nonlinear projection operators in Section \ref{section_nonlinear_families} and we formalize the definition of transversality in Section \ref{overview}. 

%TRANSITION TO TALKING ABOUT TRANSVERSALITY
 \subsection{Nonlinear projections}\label{section_nonlinear_families}
When orthogonal projections are replaced by more general families of nonlinear projection-type maps, one may ask if 
Besicovitch's theorem and its quantitative counterparts still hold. In many settings, these theorems still apply. Examples of such families include radial projections associated with visibility, curve-based projections associated with the Favard curve length and the surface projections we will introduce in this paper. Due to the special geometry exhibited by these projection families, the energy techniques of Mattila can be applied with appropriate modifications, leading to analogues lower bound on nonlinear Favard lengths.

\subsubsection{Visibility}\label{visibility_summary}
Given a point $a \in \mathbb{R}^n$, the \emph{radial projection} 
 based at $a$ maps $\mathbb{R}^n \setminus \{a\}$ to the ($n-1$)-dimensional unit sphere 
 via
\begin{equation}\label{radial} P_a(x) : = \frac{x-a}{|x-a|}.\end{equation} 
The \emph{visibility} of a measurable set $E\subset \R^n$ from a \textit{vantage point} $a$ is
\begin{equation}\label{vis} \vis(a,E) = |P_a(E)|, \end{equation}
where $|\cdot|$ denotes the $(n-1)$-dimensional Hausdorff measure on the unit sphere. In applications, we will restrict the vantage points $a$ to a \textit{vantage set}, $A$. Informally, the visibility of a set $E$ measures how much of the sky is filled up by the constellation $E$ from an observer at vantage point $a$. As such, the set $E$ is referred to as the \textit{visible set}.

  Bond, {\L}aba, and Zahl obtained upper and lower bounds on the visibility of $\delta$-neighborhoods of unrectifiable self-similar $1$-sets in the plane. In particular, their lower bound \cite[Theorem 2.4]{BLZ} for visibility states that if $\mu$ is a positive, Borel, probability measure supported on a visible set $E\subset \R^2$ paired with an $L$-shaped vantage set $A\subset \R^2$ (with an extra separation condition), then
 $$I_1(\mu)^{-1} \lesssim \int_A \operatorname{vis}(a, E) \, da.$$
Their work provides quantitative versions of the results in \cite{Mar54, SimSol}.

In this paper we will generalize this result by proving it for a wider range of vantage sets and extending it to higher dimensions. In particular, we provide a much weaker constraint on the geometric relationship between the vantage set and the visible set. As a particular application, we will demonstrate how such results can be used to obtain a lower bound on the rate of decay of the visibility of generations of the four-corner Cantor set from a wide variety of curves. 

%
%FAVARD CURVE LENGTH
\subsubsection{Favard curve length}\label{Favard_curve_warmup_section}  
As a second example of a context in which energy techniques can be applied, we define the family of maps which induce the \textit{Favard curve length}. Let $\Ga$ denote a curve in $\R^2$. Given $\alpha \in \R$ and $(x,y)\in \R^2$, let $\Phi_\al(x,y) $ denote the set of $y$-coordinates of the intersection of $(x,y)+\Ga$ with the line $\{x=\al\}$. That is,
\begin{equation}\label{phi_al}
\Phi_\al(x,y) = \{ \beta\in \R: (\al, \beta) \in \left( (x,y)+\Ga\right) \cap \{x=\al\}\}.
\end{equation}
Given $\be \in \R$, the inverse set $\Phi_\al^{-1}\pr{\be} = \set{p\in \R^2 : \be \in \Phi_\al\pr{p}}$ is given by $\pr{\al, \be} - \Ga$. In the case that $\Ga$ can be expressed as the graph of a function and $\Phi_\al(x,y) \neq \emptyset,$ then $\Phi_\al(x,y) $ is a singleton and we identify $\Phi_\al(x,y) $ with that point.
 
If $E \subset \R^2$, then the \emph{Favard curve length} of $E$ is defined by 
\begin{equation}\label{FavG}
\FavG(E)
:= |\{(\al, \be) \in \R^2 : \Phi_\al^{-1}(\be) \cap E  \ne \emptyset \}|
= \int_{\R} |\Phi_\al(E)| d\al.
\end{equation}
Our basic assumption on $\Ga$ is that it is a piecewise $\mathcal{C}^1$ curve with piecewise bi-Lipschitz continuous unit tangent vectors; these conditions will be discussed in the transversality analysis that appears in Section \ref{Favard curve section}, as well as in  Section \ref{section:nonexamples} where we consider what goes wrong for non-transversal families.

The maps under consideration were originally introduced by Simon and the second listed author of this paper to study sum sets of the form $E+\Gamma$, where $\Gamma$ denotes a sufficiently smooth curve and $E$ denotes a compact set in $ \R^2$. To see the connection, we write
\begin{align*}
\FavG(E)
&= |\{(\al, \be) \in \R^2 : \Phi_\al^{-1}(\be) \cap E  \ne \emptyset \}|\\
&= |\{(\al, \be) \in \R^2 :     \{  \pr{\al, \be} - \Ga \}  \cap E  \ne \emptyset \}|\\
&= |\{(\al, \be) \in \R^2 :      \pr{\al, \be}  \cap \left( E +  \Ga \right)  \ne \emptyset \}|\\
&= |E + \Ga|.
\end{align*}
The measure and dimension of sets of the form $E+\Ga$ was established in \cite{ST17} and the interior of such sum sets was subsequently studied in \cite{STint}.  Connections to the study of pinned distance sets and the Falconer distance conjecture are also explored there.  In both \cite{ST17} and \cite{STint}, the results rely on relating the set $E$ to the dimension, measure, and interior of the images of $E$ under the maps $\{\Phi_\al\}$. A unifying ingredient in each of these works was the observation that the maps introduced in \eqref{phi_al} are similar to  orthogonal projection maps from the prespectives of measure, dimension, and interior.

%INTERPRETATION PROBABILISTIC/ GEOMETRIC
As a further interpretation of the Favard curve length, there is a probabilistic interpretation. The Favard length of a set is comparable to its Buffon needle probability (that is, the probability that a long, thin needle dropped near the set intersects the set). In the nonlinear setting, the Favard curve length is comparable to the probability that a dropped curve meets the set -- that is, the probability that $\Ga \cap E \neq \emptyset$ after conditioning to the event that $\Ga$ lies near $E$. We denote this probability by $\mathcal{P}_\Ga(E)$. In summary, 
\begin{equation}\label{eq_formulations}
\FavG(E) \sim |E + \Ga| \sim \mathcal{P}_\Ga(E).
\end{equation}
and our Theorem \ref{main} gives a lower bound on these equivalent quantities.

Cladek, Davey, and Taylor \cite{CDT21} obtained upper and lower bounds on the Favard curve length of $\K_n$, the $n$-th generation in the construction of the four corner Cantor set:
\begin{equation}\label{CDTbounds}
\frac{1}{n} \lesssim \Fav_\Ga(\K_n) \lesssim n^{-1/6 + \delta},
\end{equation}
which by \eqref{eq_formulations} implies upper and lower bounds on $|\K_n + \Ga| \sim \mathcal{P}_\Ga(\K_n)$.
The lower bound relied on self-similarity and a square-counting argument adapted to the nonlinear setting. In this paper, we will use energy methods to provide a simple alternative proof of the lower bound in \eqref{CDTbounds} which holds in a more general setting and does not require self-similarity. See Corollary \ref{corollary:four_corner_favard} for the details.
Further, we obtain a higher dimensional analogue of the lower bound in \eqref{CDTbounds}; this is the topic of the next section.  
We return to our discussion of Favard curve length in Section \ref{Favard curve section} after stating our main results.

It is worth remarking that other authors have studied related Buffon-type probability problems. In particular, Bond and Volberg \cite{BoV11} considered lower bounds in the context of the intersection of $\K_n$ with large circles of radius $n$. In that context, the curves were adapted to the generation $n$, instead of having a fixed underlying curve.

\subsubsection{Favard surface length in $\R^d$}\label{subsection_Fav_higher}
The Favard curve length can also be formulated in a higher dimensional setting, and we refer to the resulting quantity as the \textit{Favard surface length}. 
Note that we still use the term ``length'' as we will consider a family of maps $\Phi_\al : \R^d \rightarrow \R$ and take the  average of the $1$-dimensional measures of the images of $E$ under such maps.   
To the best of the authors' knowledge, this is the first article to define such a general notion of Favard length in higher dimensions. 
\smallskip

Let $\Ga = \Ga_d$ denote a surface in $\R^d$.  
Given $\alpha \in \R^{d-1}$ and $\vec{x} =:(x_1, \cdots, x_d)\in \R^d$, let $\Phi_\al(\vec{x}) $ denote the set of $x_d$-coordinates of the intersection of $\vec{x}+\Ga$ with the line $\widetilde{x} =: (x_1, \cdots, x_{d-1}) =\al$.  
That is 
\begin{equation}\label{phi_al_d}
\Phi_\al(\vec{x}) = \{ \beta\in \R: (\al, \beta) \in (\vec{x}+\Ga) \cap \{\widetilde{x}=\al\}\}.
\end{equation}
Given $\be \in \R$, the inverse set $\Phi_\al^{-1}\pr{\be} = \set{p \in \R^{d} : \be \in \Phi_\al\pr{p}}$ is given by $\pr{\al, \be} - \mathcal{C}$. When $\Ga$ can be expressed as the graph of a function and $\Phi_\al(\vec{x}) \neq \emptyset,$ then $\Phi_\al(\vec{x}) $ is a singleton and we identify $\Phi_\al(\vec{x}) $ with that point.

If $E \subset \R^d$, then the \emph{Favard surface length} of $E$ is defined by 
\begin{equation}\label{FavG_d}
\Fav_{\Ga,d}(E)
:= |\{(\al, \be) \in \R^{d} : \Phi_\al^{-1}(\be) \cap E  \ne \emptyset \}|
= \int_{\R^{d-1}} |\Phi_\al(E)| d\al.
\end{equation}
As was the case for the Favard curve length defined in the previous section, the Favard surface length of a set $E$ is 
equivalent to the $d-$dimensional Lebesuge measure of the Minkowski sum:
$$ \Fav_{\Ga,d}(e) \sim |E+\Ga|_d.$$ 
The quantity $\Fav_{\Ga,d}(E)$ has a probabilistic interpretation in terms of a Buffon \textit{surface} problem. 

\subsection{Overview of transversality}\label{overview} 
It is known that nonlinear analogues of Besicovitch's and Marstrand's projection theorems hold for families of maps satisfying a \textit{transversality condition}. A version of the Besicovitch projection theorem for transversal families can be found in \cite{HovJ2Led}, and a quantitative version is developed in \cite{DT21}. Marstrand's theorem is developed in the transversal setting in \cite[Theorem 5.1]{Sol98} and \cite[Chapter 18]{Mat15}; see also Proposition \ref{thm:basic_marstrand}.

The concept of transversality originated from the work of Simon and Pollicott \cite{PoSi}, where it was used to study the Hausdorff dimension of the attractors of a one-parameter family of IFS (iterated function systems). Solomyak then developed the transversality condition for the absolute continuity of invariant measures for a one parameter family of IFS in \cite{Sol95}.
Moreover, Solomyak combined the methods from \cite{PoSi} and \cite{Sol95} in \cite{Sol98}
to establish a much more general transversality method for generalized projections. 
The next step was made by Peres and Schlag \cite{PeSc00}, who further developed the method of transversality and gave a number of far reaching applications. Such results have been utilized and further developed by a  number of authors with far reaching geometric applications.  See, for instance, \cite{Bourgain10}, \cite{CDT21}, \cite{PeSc00}, \cite{Shmerkin20}, \cite{ST17}, \cite{STint}. 

 The transversality condition naturally arises when studying projection-type operators that do not overlap too much with each other, and this paper will explore the role transversality plays in developing energy estimates. The transversality condition addresses how, for distinct points $x$ and $y$ in the plane, the graphs $\left\{ (\theta, \pi_\theta(x)) \right\}$and $\left\{ (\theta, \pi_\theta(y)) \right\}$ should behave at points of intersection. 
Roughly speaking, it says that if $\pi_\theta(x)$ and $\pi_\theta(y)$ are close for some value of $\theta$, then they cannot remain close as $\theta$ changes. That is, the graphs cannot intersect tangentially, but must do so at a positive angle.

An alternative perspective on transversality will frequently come up in our techniques. If $x$ and $y$ are two fixed points, then the set of projections which cannot distinguish $x$ and $y$ must be rather small; placing this on the appropriate scale, this means that for each $\delta > 0$ there is an upper bound on the size of the set
$$\left\{\theta : \frac{|\pi_{\theta}(x) - \pi_{\theta}(y)|}{|x - y|} \le \delta\right\}.$$
Informally, this means that if $\pi_{\theta}$ is a randomly chosen projection then it will, with high probability, separate $x$ and $y$ on the projection side.

We now make precise our notion of transversality. The main objects are an indexed family of maps, a common domain and codomain equipped with measures, and a probability measure on the index set. In Section 2, we will place each of the families mentioned previously in the context of this definition and establish transversality with the appropriate parameters. 

\begin{defn}[Nonlinear projections]\label{assumption:big} For $1\le m< n$, a \textit{family of projection-type operators} will have the following objects associated to it:
\begin{itemize}
    \item a domain $\Omega$ contained in $\R^n$
    \item a codomain $X$ contained in a Euclidean space,
     a nonnegative integer $m$, and a {Borel} measure $h$ on $X$ such that
    $$h(B(x, \delta)) \gtrsim \delta^m$$
    for all $x \in X$ and $\delta \in (0, 1)$,
    \item an indexing set $A$ contained in an Euclidean space equipped with a compactly supported probability measure $\psi$,
    \item and a family of maps $\proj : \Omega \to X$ indexed by $\alpha \in A$ such that the function $(p, \alpha) \mapsto \proj(p)$ is continuous.
\end{itemize}
\end{defn}
In order to be transversal, we will require that the family of projections satisfies a compatibility condition for different parameters:
\begin{defn}[Transversality]\label{definition:s-transversal}
For a given $s \ge 0$, a family of maps $\{\proj : \alpha \in A\}$ satisfying Definition \ref{assumption:big} is called \emph{$s$-transversal} if there exists constants $c>0$ and $\de_0>0$ so that, for all distinct $x, y \in \Omega$ and $0<\de \le \de_0$, we have 
\begin{equation}\label{H2alt}
\psi\{\alpha : |\proj(x) - \proj(y)| \le \delta |x - y|\} < c\cdot \delta^m \cdot |x - y|^{m - s},
\end{equation}
or equivalently that
\begin{equation}\label{H2}
\psi\{\alpha : |\proj(x) - \proj(y)| \le \de\} < c \cdot\frac{\de^m}{|x - y|^s}.\end{equation}
\end{defn}

Although this definition is written with a tunable parameter $s$, our most important case will be when the parameter $s$ for transversality matches the dimension $m$ of the target space; in this case, the transversality condition reduces to
$$\psi \{\alpha : |\proj(x) - \proj(y)| \le \delta |x - y|\} \lesssim \delta^m.$$
We note that our definition has some points in common with Mattila's definition in \cite[Definition 18.1]{Mat15}, but that we do not require smoothness of the projections nor derivative bounds of non-zero order.

\subsection{Main results}\label{section_main}
The key uniting theme of our results is that for families of maps satisfying the transversality condition introduced in Definition \ref{definition:s-transversal}, the energies associated to a measure $\mu$ will be closely related to the energies of the pushforward measures $\widetilde{\pi_{\alpha}}_{\sharp} \mu$. As a demonstration of the techniques, we will begin by giving a brief formulation of part of the Marstrand projection theorem in the transversal setting: the dimension of a typical projection of a set with dimension $s < 1$ does not decrease. The proof of this fact, found in Section \ref{energy_section}, demonstrates the utility of examining the energy of pushforward measures and is similar to the presentation in \cite[Chapter 18]{Mat15}. (For the statement of the Marstrand projection theorem in the classic setting for orthogonal projections, as well as a formulation in higher dimensions, see \cite[Section 5.3]{Mat15}.)

\begin{prop}[Nonlinear Marstrand theorem]\label{thm:basic_marstrand}
Suppose that $\{\proj : \alpha \in A\}$ is a family of maps into an $m$-dimensional space supporting a measure $h$, as in Definition \ref{assumption:big}. If $E$ is a set with Hausdorff dimension $t \le m$ and the family of projections is $m$-transversal, then for $\psi$-almost every $\alpha \in A$ we have
\begin{equation}\label{first_part_mar} \dimh \proj E = t.\end{equation}
\end{prop}

Developing the energy techniques further, we give more general asymptotic lower bounds on the average size of a projection. The next theorem serves as a direct generalization of Mattila's result Theorem \ref{theorem:mattila90_main}.

\begin{thm}[Average nonlinear projection length for neighborhoods]\label{main}
With the notation of Definition \ref{assumption:big}, assume
 that $\{\proj : \alpha \in A\}$ is an $m$-transversal family of projections into an $m$-dimensional space. 
 Fix a positive Borel probablity measure $\mu$ supported on a compact set $F \subseteq \Omega$, so that 
 $$\mu(B(x, r)) \lesssim r^t$$
for all $x \in \Omega$ and $0 < r < \infty$. 
\begin{itemize}
    \item If $t < m$, then
    $$\int_A h(\proj F(r)) \, d\psi(\alpha) \gtrsim r^{m - t}.$$
    \item If $t = m$, then
    $$\int_A h(\proj F(r)) \, d\psi(\alpha) \gtrsim (\log r^{-1})^{-1}.$$
\end{itemize}
\end{thm}

As a first application, we can phrase Theorem \ref{main} in the setting of radial projections and visibility defined in \eqref{radial} and \eqref{vis} respectively. 
\begin{thm}[Visibility for surfaces in $\R^n$]\label{theorem:smooth_surface_visibility}
Fix a set $E \subseteq \mathbb{R}^n$ of positive and finite $s$-dimensional Hausdorff measure, and consider a vantage set $A$ which is a piecewise smooth $(n-1)$-dimensional surface equipped with Hausdorff measure; assume that for all $a \in A$ and $e\in E$ we have $|a - e| \lesssim 1$. 
Finally, assume that there exists a positive $\rho$ such that for almost every $a \in A$ the tangent plane based at $a$ does not pass within distance $\rho$ of $E$.  The following statements hold:
\begin{itemize}
\item the family of radial projections $\{P_a : a \in A\}$ is $(n-1)$-transversal,
\item if $s < n - 1$ we have
$$\int_{A} \operatorname{vis}(a, E(r)) \, d\mathcal{H}^{n-1}(a)  \gtrsim r^{n - 1 - s},$$
\item and if $s = n - 1$ we have
$$\int_{A} \operatorname{vis}(a, E(r)) \, d\mathcal{H}^{n-1}(a) \gtrsim (\log r^{-1})^{-1}.$$
\end{itemize}
\end{thm}
The first claim of Theorem \ref{theorem:smooth_surface_visibility} is established in Section \ref{section_vis_check_trans} and the latter two claims are established in Section \ref{section_apps}.

In a similar manner, we can put this result in the context of Favard curve length defined in \eqref{FavG}. For curves in the plane, our techniques yield the following:
\begin{thm}[Favard curve length of neighborhoods]\label{FavTheorem}
Let $E$ be a compact set in the plane and $\Gamma$ a piecewise $\mathcal{C}^1$ curve with piecewise bi-Lipschitz continuous unit tangent vectors. Assume further that $E$ supports a Borel probability measure $\mu$ with the $t$-dimensional growth condition $\mu(B(x, r)) \lesssim r^t$ for all $x \in E, 0 < r < \infty$. The following statements hold:
\begin{itemize}
\item the family of curve projections $\Phi_{\alpha}$ is $1$-transversal,
\item if $t < 1$, then for all sufficiently small $r$ we have
$$\FavG(E(r)) \gtrsim r^{1 - t}$$
\item and if $t = 1$, then for all sufficiently small $r$ we have
$$\FavG(E(r)) \gtrsim (\log r^{-1})^{-1}.$$
\end{itemize}
\end{thm}

Next, we consider applications of Theorem \ref{main} to study self-similar sets such as $\mathcal{K}_n$, the $n$-th generation in the construction of the four corner Cantor set. Although they are not precisely the same as neighborhoods of $1$-sets, the sets $\K_n$ still support measures with easily computable density and Mattila's energy techniques can be adapted to estimate their visibilities \eqref{vis} and Favard curve lengths \eqref{FavG} from below. Our techniques are similarly amenable to such sets, and we will have the following corollaries:

\begin{cor}[Visibility of $\K_n$]\label{corollary:four_corner_visibility}
Suppose that $\Gamma$ is a smooth curve such that for any point $x \in [0, 1]^2$ and any $\gamma \in \Gamma$ we have $|x - \gamma| \sim 1$, and that no tangent line to $\Gamma$ passes through $[0, 1]^2$. Then
$$\int_{\Gamma} \vis(a, \mathcal{K}_n) \, d\mathcal{H}^1(a) \gtrsim \frac 1 n.$$
\end{cor}

\begin{cor}[Favard curve length of $\K_n$]\label{corollary:four_corner_favard} If $\Gamma$ is a piecewise $\mathcal{C}^1$ curve with piecewise bi-Lipschitz continuous unit tangent vectors, then
$$\FavG(\mathcal{K}_n) \gtrsim \frac 1 n.$$
\end{cor}

Although these results are stated for the generations $\mathcal{K}_n$ specifically, there are substantial generalizations of the results. The core fact used in the proof is that $\mathcal{K}_n$ supports a measure with a specific density property; this behavior can be observed in a very broad family of $1$-dimensional fractal sets generated by iterated function systems.

Finally, we consider an application of Theorem \ref{main} for the Favard surface length, defined in \eqref{FavG_d}, when $d=3$. Although we do not state them here, there are natural generalizations of this result to arbitrary dimension.
\begin{thm}[Favard surface length of neighborhoods]\label{FavTheorem3d}
Let $E$ be a compact set in the plane and $\Gamma$ 
denote a surface in $\R^3$ defined by $\Ga = \{(t, \ga(t)): t\in I  \}$, where 
$\ga: \R^2 \rightarrow \R$, $\ga(s) = f(|s|)$, and $f: \R \rightarrow \R$ is a $C^2$ function on a non-empty compact interval $I$ satisfying $f(x) = f(-x)$, with $f''>0$ on $I$. Assume further that $E$ supports a Borel probability measure $\mu$ with the $t$-dimensional growth condition $\mu(B(x, r)) \lesssim r^t$ for all $x \in E, 0 < r < \infty$. The following statements hold:
\begin{itemize}
\item the family of curve projections $\Phi_{\alpha}$ is $1$-transversal,
\item if $t < 1$, then for all sufficiently small $r$ we have
$$\FavG(E(r)) \gtrsim r^{1 - t}$$
\item and if $t = 1$, then for all sufficiently small $r$ we have
$$\FavG(E(r)) \gtrsim (\log r^{-1})^{-1}.$$
\end{itemize}
\end{thm}

The outline of the paper is as follows. In Section \ref{section:trans}, we will show how each of the aforementioned families of maps exhibit the required transversality properties.  Geometrically motivated proofs are given for each family. Section \ref{energy_section} develops the energy techniques necessary to study pushforward measures, beginning with an illustration of how a transversal family of maps can be used to prove a classical result of Marstrand. The proof of Theorem \ref{main}  appears in Section \ref{energy_section}. In Section \ref{section_apps}, we prove Theorems \ref{theorem:smooth_surface_visibility} and \ref{FavTheorem} as applications of Theorem \ref{main} paired with the transversality established in Section \ref{section:trans}, and we explore applications and sharpness examples.

\section{Establishing Transversality}\label{section:trans}
The aim of this section is to illustrate several families of projections that meet the transversality condition described in Definition \ref{definition:s-transversal}. This includes orthogonal, radial, curve, and surface projections.

\subsection{Orthogonal projections}
 Our first example of a transversal family is the collection of orthogonal projections from $\mathbb{R}^n$ to $\mathbb{R}^m$ for some $m < n$. To be explicit about the setup, we will consider a domain $\Omega = \mathbb{R}^n$, a codomain $X = \mathbb{R}^m$, and equip the codomain with the appropriate Lebesgue measure. We then have the family 
$$\{\iota_V \circ P_V : V \in G(n, m)\}$$
of projections indexed by the Grassmanian, where $P_V$ is the orthogonal projection into the $m$-plane $V$, and with the natural inclusion $\iota_V : V \to \mathbb{R}^m$; equip this set with the Haar measure $\gamma_{n, m}$. The full details of the construction of the Grassmanian manifold and the measure $\gamma_{n, m}$ can be found, for example, in \cite[Chapter 3]{Mat95}.

For establishing transversality, the core estimate in this context is contained in \cite[Lemma 2.7]{Mat95}: for any distinct points $x, y \in \mathbb{R}^n$,
\begin{equation} \label{equation:lemma27}
\gamma_{n, m} \left(\left\{V \in G(n,m) : |P_V(x - y)| \le \delta\right\} \right) \sim \frac{\delta^m}{|x - y|^m}
\end{equation}
Using the linearity of $P_V$, one can quickly establish 
\begin{lem}[Orthogonal projections are transversal]\label{lemma:orthogonal_transversal}
The family of orthogonal projections from $\mathbb{R}^n$ to $\mathbb{R}^m$ equipped with the Haar measure $\gamma_{n, m}$ is $m$-transversal.
\end{lem}

As in \cite{Mat95}, this can be done geometrically, by reducing to an estimate of the $m$-dimensional measure of a patch on a sphere. There is also an important probabilistic interpretation, which will turn out to be the main ingredient when studying other transversal families. If $x$ and $y$ are fixed points in $\mathbb{R}^n$, then a randomly chosen $m$-dimensional plane is likely to preserve some, if not most, of the distance between $x$ and $y$; that is, on average we have that $|P_V(x) - P_V(y)| \ge \delta |x - y|$. However, there is still an exceptional set of $m$-planes which do not respect this inequality at scale $\delta$ -- for example, any $m$-plane which is sufficiently close to lying in the orthogonal complement to the line between $x$ and $y$. Transversality comes from controlling the $\gamma_{n, m}$-measure of the exceptional set for scale $\delta$.

\subsection{Visibility}\label{section_vis_check_trans}
We now turn to establishing the transversality condition for families of radial maps. We begin by first recalling the notation defined in Section \ref{visibility_summary}.   
For a point $a$ in $\mathbb{R}^n$, the radial projection based at $a$ maps $\mathbb{R}^n \setminus \{a\} \to \mathbb{S}^{n - 1}$ via
$$P_a(x) = \frac{x-a}{|x-a|}.$$
For a fixed vantage set $A\subset \mathbb{R}^n$ equipped with a measure $\psi$, our family of projections will be $\{P_a : a \in A\}$. The common domain will be a visible set $E$, which will be assumed to be disjoint from $A$.  Our codomain is $\mathbb{S}^{n - 1}$ equipped with the surface measure and so $m= n - 1$ and $P_a:  E\rightarrow S^{n-1}$.
The aim of this section is to 
establish some minimal geometric relations between the vantage set $A$ with the measure $\psi$ and the visible set $E$ so that the family $\{P_a: a\in A\}$ is $(n-1)$-transversal.  A natural condition on the probability measure $\psi$ will arise after we analyze the geometry of the radial projections. 

To this end, we will make use of the following geometric lemma. A two-dimensional variant appeared in the work of Bond, {\L}aba, and Zahl \cite[Lemma 2.3]{BLZ}; we will provide a somewhat different proof and generalize the result to higher dimensions.

\begin{lem}[Visibility and tubes]\label{lemma:visibility_geometry}
Fix a scale $R>0$ and two points $x, y$ not contained in the vantage set $A$ with $|x - y| \le R$. Let $L_{x, y}$ denote the line connecting them. Then there exists a constant $C < \infty$ depending only on $R$ such that
$$\{a\in  A  : |P_a(x) - P_a(y)| \le \delta |x - y|\} \cap B(x, R) \subseteq L_{x, y}(C \delta),$$
where $ L_{x, y}(C \delta)$ denotes the $C\de$-neighborhood of the line $L_{x,y}$. 
\end{lem}

\begin{proof}
We proceed by contrapositive. Suppose that $a$ is within the ball $B(x, R)$ but outside the tube $L_{x, y}(\rho)$ of radius $\rho$ around $L_{x, y}$. Draw a  triangle with vertices $x, y$, and $a$; let $\theta$ denote the internal angle at vertex $a$ and $\gamma$ denote the internal angle at vertex $y$. 
%By the geometry of the radial projection, it is immediate to verify that 
Since $|P_a(x) - P_a(y)|$ is comparable to the internal angle $\theta$ of the triangle, it is sufficient to give a lower bound on the angle $\theta$. By the law of sines, we have that
$$\frac{\sin \theta}{|x - y|} = \frac{\sin \gamma}{|a - x|}$$
so that
$$\theta \ge \sin \theta = \frac{\sin \gamma}{|a - x|} |x - y|.$$
If $\mathcal{A}_{x, y, a}$ denotes the altitude of the triangle (as viewed with base side $\overline{xy}$) then
$$\sin \gamma = \frac{\mathcal{A}_{x, y, a}}{|y - a|}$$
and
$$\theta \geq \frac{\mathcal{A}_{x, y, a} \cdot |x - y|}{|a - x| \cdot |a - y|}.$$

Since $a, x, y \in B(x, R)$, we have that $|a - x| \le 2R$ and $|a - y| \le 2R$. 
 Since $a$ lies outside the tube $L_{x, y}(\rho)$, the altitude must be at least $\rho$. Therefore, there exists a constant $c \sim 1$ for which
$$|P_a(x) - P_a(y)| \ge c \theta \geq c \cdot  \frac{\rho}{4R^2} \cdot |x - y|.$$

Choosing $\rho = C\delta$ 
for $C>4R^2/c$
establishes that $|P_a(x) - P_a(y)| > \delta |x - y|$, as desired.
\end{proof}

We now have a natural condition to impose on the probability measure $\psi$: as we wish to verify \eqref{H2alt} with $s=m=n-1$, then Lemma \ref{lemma:visibility_geometry} implies that the measure of a tube should be bounded by the radius of the tube to an appropriate power. 
To be precise, we will say that $\psi$ satisfies the \textbf{tube condition with respect to $E$} if for any tube $T_{\delta}$ with sufficiently small radius $\delta$ that passes through the visible set, $E$, we have
\begin{equation}\label{equation:tube}
\psi(T_{\delta}) \lesssim \delta^{n - 1}.
\end{equation}
In this case, we have established that, provided the distance from $A$ to $E$ is at most $R$,  $\{P_a :a \in A\}$   is a family of maps from an $n$-dimensional space to an $(n - 1)$-dimensional space with
$$\psi\{a \in A : |P_a(x) - P_a(y)| \le \delta|x - y|\} \lesssim \delta^{n - 1}.$$
Comparing this to the definition of transversality, we have established the following:

%GENERAL LEMMA
\begin{lem}[Radial maps are transversal]
\label{lemma:tube_transversal}
Fix a scale $R>0$.  Fix a vantage set $A$ and a visible set $E$ with the condition that for all $a \in A$ and $e\in E$
we have $|a - e| \lesssim 1$. 
If $A$ is equipped with a measure $\psi$ satisfying the tube condition with respect to $E$ \eqref{equation:tube}, then the family $\{P_a : a \in A\}$ is $(n-1)$-transversal as in \eqref{H2}. 
\end{lem}

This gives a substantial degree of flexibility in structuring the vantage set. One application of this technique is to a vantage set which is made up of a smooth curve $\Gamma$ whose tangent lines do not come too close to the visible set. When $\psi$ is taken to be the restriction of $\mathcal{H}^{n - 1}$ to the vantage set $A$, this will imply that $\psi$ satisfies the tube condition with respect to $E$. We discuss this idea more in Section \ref{section:main_proofs}.

%%%%%%%%%%%%%%%%%%%%%%%%%%%%%%%%%%%%%%%%%%%%%%
\subsection{Favard curve length}\label{Favard curve section}
In this section, we verify that the family of maps $\Phi_\la:\R^2 \rightarrow \R$ introduced in \eqref{phi_al} satisfy the transversality condition of Definition \ref{definition:s-transversal}. 
This will proceed through a couple of reductions. First, we will set up some basic assumptions on the smoothness of the curve as well as some notation.  
Next, by breaking the curve into simpler pieces, we reduce to the case of a curve that is a graph satisfying a simpler curvature condition. 
We establish transversality in this simpler setting and note  this is sufficient to establish lower bounds on the Favard length for the general setting. 

% ASSUMPTIONS ON THE CURVE
\begin{defn}\label{defn:curvature}
We say that $\Ga$ satisfies our \emph{standard curvature condition} if $\Ga$ is a piecewise $\mathcal{C}^1$ curve with piecewise bi-Lipschitz continuous unit tangent vectors.
\end{defn}

Under the assumptions of Definition \ref{defn:curvature}, $\Ga$ can be expressed as a disjoint union of continuous subcurves $\Ga = \bigcup_{i=1}^{\infty} \Ga_i$, where each $\Ga_i$ is $\mathcal{C}^1$ of finite length with a bi-Lipschitz continuous unit tangent vector.
By further decomposition of the curve, each $\Ga_i$ can be expressed either as a graph with respect to the first coordinate, $\Ga_i= \{(t, \ga_i(t)): t\in I_i\}$, or as a graph with respect to the second coordinate, $\Ga_i= \{(\ga_i(t), t): t\in I_i\}$, so that $\sup_{t\in I_i}|\ga'_i (t) | \le 1$, and $\ga'_i$ 
is $\la_i$-bi-Lipschitz.

In order to obtain lower bounds on $\Fav_{\Ga}(E)$, where $E$ will denote a compact subset of $\R^2$, 
since $\Fav_\Ga(E) \geq \Fav_{\Ga_i}(E)$ for each $i$,
 it suffices to obtain lower bounds on $\Fav_{\Ga_i}(E)$.
Fixing $i$ and observing that 
rotating the curve and the set $E$ by the same amount has no affect on $\Fav_{\Ga_i}(E) = |E + \Ga_i|$, 
we may simply assume that $\Ga_i$ is a graph with respect to the first coordinate. Finally, for ease of notation, we drop the subscript $i$ and assume that $\Ga$ has all the properties of $\Ga_i$.  

\begin{defn}\label{defn:curvature_simple}
 We say that $\Ga$ is a curve satisfying the \emph{simple curvature condition} if $\Ga= \{(t, \ga(t)): t\in I\}$, where $\ga:\R\rightarrow \R$, 
 \begin{equation}\label{deriv_cond}
 \sup_{t\in I}|\ga'(t) | \le 1,
 \end{equation}
 and 
$\ga'$ 
is $\La$-bi-Lipschitz satisfying 
\begin{equation}\label{Lips_cond}
\La^{-1} |s-t| \le |\ga'(s) - \ga'(t)| \le \La |s-t|
\end{equation}
for some $0<\La<\infty$ and for each $s,t$ in a non-trivial closed interval $ I$.
 \end{defn}

Let $\Ga= \{(t, \ga(t)): t\in I\}$ be a curve satisfying the simple curvature condition of Definition \ref{defn:curvature_simple}.  
Note that condition \eqref{Lips_cond} guarantees that $\ga'$ is monotonic; without loss of generality, we will assume that $\Ga$ is concave down so that if $t<s$, then 
\begin{equation}\label{convex_cond}
\frac{ \ga'(s) - \ga'(t)  }{s-t } <0.
\end{equation}
Write $I=[L_1, L_2]$ for some $L_1< L_2$ and set $h= \frac{L_2- L_1}{2}$. 
Set $\Om= [0, h]^2 \subset \R^2$
and
$A=\left[L_1 + h, L_2 \right].$
With this set up, for each $\la \in A$ and $\mathbf{a} \in \Om$, 
$$\ell _\la\cap(\mathbf{a}+\Gamma)  = (\la, a_2 + \ga(\la - a_1) ) $$ is a singleton, as in Figure \ref{a11}, and we can define the one-parameter family of mappings $\{\Phi_\la(\mathbf{ \cdot })\}_{\la\in A}$, 
$\Phi_\la:\Om\to \ell _\la$
 by
 \begin{equation}\label{remind_phi}
\Phi_\la(\mathbf{a}) = a_2 + \ga(\la - a_1).
 \end{equation}

\begin{figure}[h]
  \centering
  \includegraphics[width=6cm]{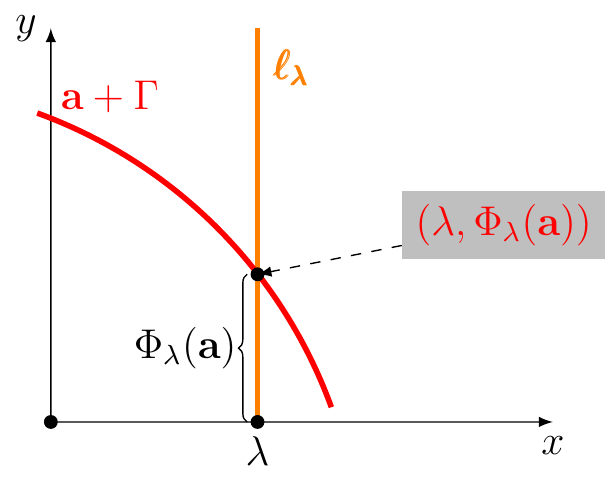}
\caption{$\Phi_\lambda(\mathbf{a})$}\label{a11}
\end{figure}

We are now ready to show that the simple curvature assumption implies $1$-transversality.  In line with Definition \ref{assumption:big}, our codomain is $\R$ equipped with the $1$-dimensional Lebesgue measure and so $m=1$.

\begin{lem}[Curve maps are transversal]\label{curve_trans}
Let $\Gamma$ be a curve satisfying the simple curvature assumption of Definition \ref{defn:curvature_simple}. 
Equip the parameter space $A$ with the $1$-dimensional Lebesgue measure. 
Then the associated family of projections  $\{\Phi_{\la}:\Om\rightarrow \R : \la \in A\}$
is $1$-transversal as in \eqref{H2}. 
\end{lem}
 
\begin{proof}
Fix a choice of $\mathbf{a}=(a_1,a_2),\mathbf{b}=(b_1,b_2)\in \Om$ with $\mathbf{a} \neq \mathbf{b}$. The proof comes in two parts: the translated graphs $(\mathbf{a} + \Gamma)$ and $(\mathbf{b} + \Gamma)$ will either intersect at a point, or they will be disjoint. We first handle the intersecting case when
\begin{equation}\label{007}
  (\mathbf{a}+\Gamma)\cap(\mathbf{b}+\Gamma)\ne \emptyset .
\end{equation}
That is, suppose there exist $s_0,t_0\in I$ and $\mathbf{a}=(a_1,a_2)\in \R^2$ such that
$$
  \mathbf{x}:=(a_1,a_2)+(s_0, \gamma(s_0)) = (b_1,b_2)+ (t_0, \gamma(t_0)).
$$

\begin{figure}[h]
  \centering
  \includegraphics[width=12.5cm]{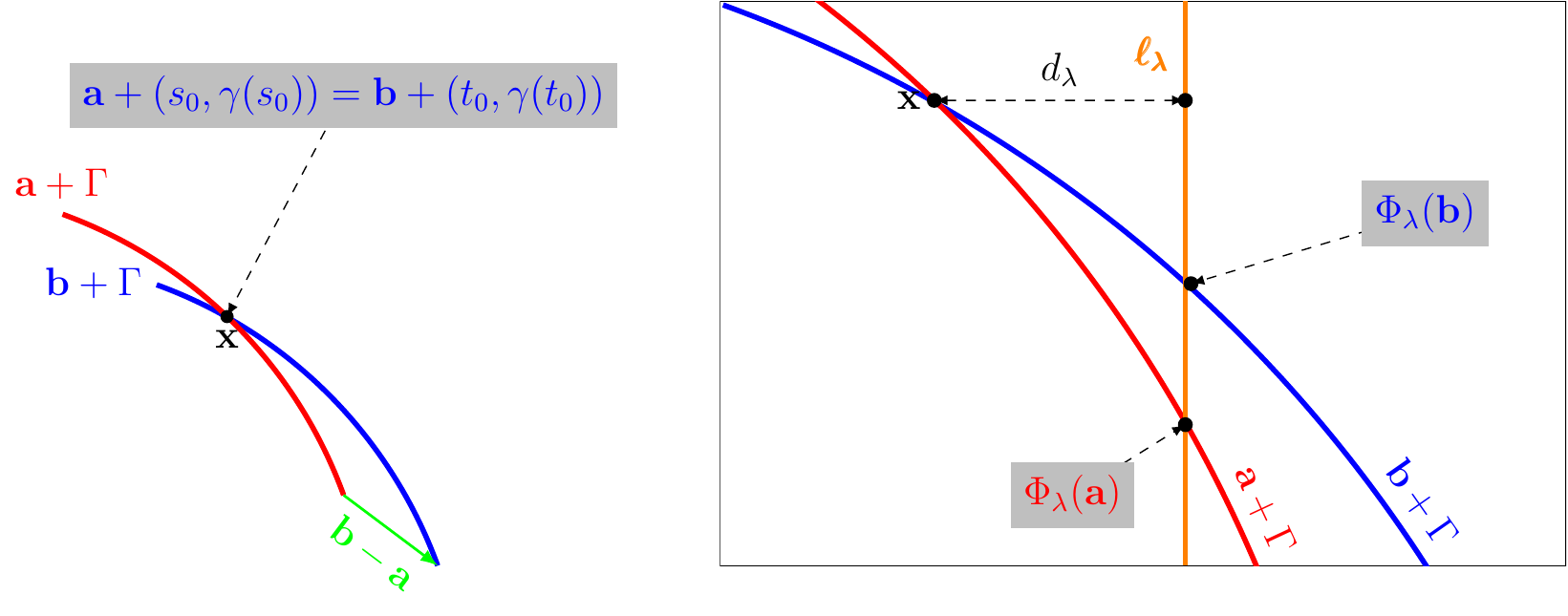}
  \caption{}\label{figure_2_c_a-vertical}
\end{figure}
Comparing coordinates, we have
 \begin{equation}\label{eq1}
 x_1= a_1  + s_0 = b_1 + t_0
 \end{equation}
 and 
 \begin{equation}\label{eq2}
 x_2 = a_2 + \gamma(s_0) = b_2 + \gamma(t_0).
  \end{equation}

For $\la \in A$, set
\begin{equation}\label{distance_la}
d_\la : = \mathrm{dist}(\mathbf{x},\ell _\lambda)=|\lambda - x_1|,\end{equation}
as depicted in Figure \ref{figure_2_c_a-vertical}. 
We verify that
\begin{equation}\label{transv_goal}
|\Phi_\lambda(\mathbf{a})-\Phi_\lambda(\mathbf{b})|
\sim
d_\lambda \cdot |\mathbf{a}-\mathbf{b}|
\end{equation}
where the implied constant is independent of $\lambda$, $\mathbf{a}$, and $\mathbf{b}$. Strictly speaking, we only need that the left hand side dominates the right hand side. Upon establishing equation \eqref{transv_goal}, it will follows that if $\de>0$ and $\lambda\in A$ satisfy
$|\Phi_\lambda(\mathbf{a})-\Phi_\lambda(\mathbf{b})|\le \de$, then
$$d_\lambda \cdot |\mathbf{a}-\mathbf{b}|\lesssim \de,$$
and so
\begin{equation}\label{1-trans}
|\{\la\in A: |\Phi_\lambda(\mathbf{a})-\Phi_\lambda(\mathbf{b})|\le \de\}|
\lesssim \frac{\de}{|\mathbf{a}-\mathbf{b}|}
\end{equation}
which is the desired transversality condition.

%the result of $1$-transversality will be established when $\mathbf{a}$ and $\mathbf{b}$ are as in \eqref{007}.

We have two further reductions. First, as depicted in Figure \ref{figure_2_c_a-vertical}, we consider the case when $\la \geq x_1$ so that
\begin{equation}\label{distance_la_updated}
d_\la =\la - x_1 \geq 0.
\end{equation}
Note that the case when when $\la - x_1 < 0$ can be handled by reflecting $E$ and $\Ga$ about the $y$-axis. Secondly, by relabeling $\mathbf{a}$ and $\mathbf{b}$ if necessary, we may assume that when $\la> x_1$, we have
\begin{equation}\label{positivity_cond_phi} 
\Phi_{\lambda}(\mathbf{b}) - \Phi_{\lambda}(\mathbf{a})>0
 \end{equation}
as in Figure \ref{figure_2_c_a-vertical}. Finally, we will also have
\begin{equation}\label{positivity_cond_space}
(b_1-a_1)>0.
\end{equation}
This follows from the geometry of the curves: in order for \eqref{positivity_cond_phi} to hold in the intersecting case, the convexity of $\Gamma$ shows that $\mathbf{b}$ must lie below and to the right of $\mathbf{a}$. 

Using the convexity condition \eqref{convex_cond}, we will show that 
\begin{equation}\label{transv_goal_mini}
 \Phi_{\lambda}(\mathbf{b}) - \Phi_{\lambda}(\mathbf{a}) \sim (b_1-a_1)\cdot d_\la.
 \end{equation} Observe that by the bound on $\gamma'$ and the relationships established in \eqref{eq1} and \eqref{eq2},
\begin{equation}%\label{coordinate_cond}
|b_2-a_2| = |\ga(s_0) - \ga(t_0)| \le |s_0-t_0| = |b_1-a_1|.
\end{equation}
As such, proving \eqref{transv_goal_mini} will be sufficient to establish \eqref{transv_goal}. We now carry out the verification of \eqref{transv_goal_mini} in three cases based on the relative sizes of $d_{\lambda}$ and $|b_1 - a_1|$. We will handle the non-intersecting case (where \eqref{007} does not hold) separately.

\smallskip

\noindent
\textbf{Case 1:} $(b_1 - a_1) < \frac{d_{\lambda}}{2}$. 
We begin by examining the simplest case, which motivates the finer analysis to come. This is depicted in the following figure:
\begin{figure}[ht]
  \centering
  \includegraphics[width=12cm]{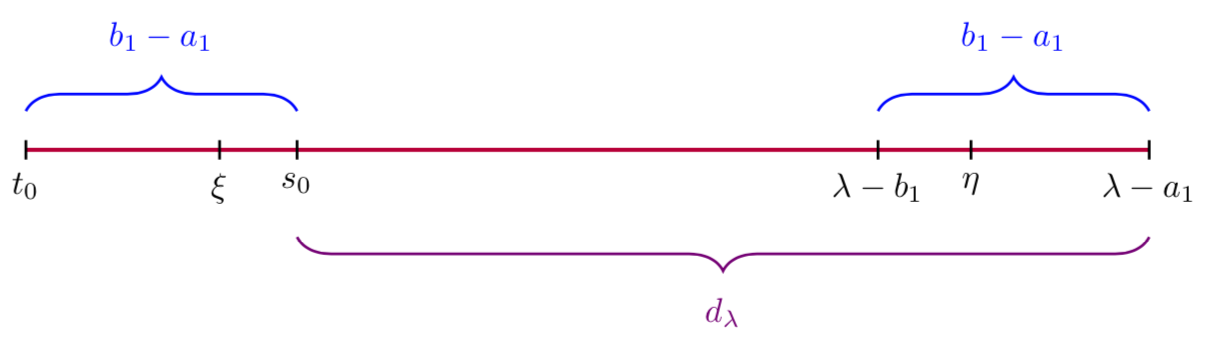}
  \caption{Case 1}
  \label{figure_case1}
\end{figure}

Using the relationships established in \eqref{remind_phi} -- \eqref{eq2} and the mean value theorem, we have
\begin{align*}
\Phi_{\lambda}(\mathbf{b}) - \Phi_{\lambda}(\mathbf{a}) 
&= \left( b_2 +  \gamma(\lambda-b_1)\right)  -    \left( a_2 +  \gamma(\lambda-a_1) \right)  \\
&=     \left(  b_2 -a_2\right)     + \left( \gamma(\lambda-b_1) - \gamma(\lambda-a_1)\right)   \\
&=      \left(  \gamma(s_0) - \gamma(t_0)\right)  + \left(  \gamma(\lambda-b_1) - \gamma(\lambda-a_1)\right)  \\
&=   \gamma'(\xi)(b_1 - a_1) -   \gamma'(\eta)(b_1 - a_1) \\
&= [\gamma'(\xi) - \gamma'(\eta)](b_1 - a_1),
\end{align*}
for some $\eta \in (\lambda- b_1,\lambda - a_1)$ and $\xi\in (t_0, s_0)$.

It follows by \eqref{Lips_cond} and \eqref{convex_cond} that 
$$\Phi_{\lambda}(\mathbf{b}) - \Phi_{\lambda}(\mathbf{a})  \sim 
 (\eta - \xi) \cdot (b_1- a_1).$$
Since $(b_1-a_1)< d_\lambda/2$, we see that  \eqref{transv_goal_mini} is verified following the observation that
$$(\eta - \xi ) \sim d_\lambda.$$
To see this, recall from \eqref{distance_la_updated} and \eqref{eq1} that 
$d_\la = \la- a_1 - s_0 = \la - b_1 -t_0$. Following Figure \ref{figure_case1},
$$\eta - \xi > (\lambda - b_1) -s_0 = (\lambda- b_1 - t_0) - (s_0 - t_0) = d_\lambda - (b_1 - a_1),$$
and similarly  
$$\eta - \xi < (\lambda-a_1) -t_0 = (\lambda-a_1-s_0) + (s_0 -t_0) = d_\lambda + (b_1 - a_1).$$

%%%VERIFY  \eqref{transv_goal}: The FULL INTERSECTION CASE
Before moving to the general argument, we observe that the separation of $d_\la$ and $(b_1-a_1)$ was crucial in guaranteeing that the variables arising from the application of the mean value theorem, $\xi$ and $\eta$, were properly separated.  More generally, a finer analysis using telescoping sums is used to guarantee such separation. 
\smallskip

%%%%%%%%%%%%%%%%%%%%%%%%%%%%%%%%%%%%%%%%%%%%%%%%%%
%CASE 2
%%%%%%%%%%%%%%%%%%%%%%%%%%%%%%%%%%%%%%%%%%%%%%%%%%
\noindent
\textbf{Case 2:} $\frac{d_{\lambda}}{2} \le (b_1 - a_1 )< d_{\lambda}$. 
Set 
\begin{equation}\label{p_k}
p = \frac{b_1- a_1}{2} , \,\ \text{ and } \,\, 
q= s_0.
\end{equation}

First, 
we take a moment to compare the variables under examination. 
Note $p>0$ by \eqref{positivity_cond_space}. 
Using \eqref{distance_la_updated} and \eqref{eq1}, we can write $d_\la = \la - b_1 - t_0$ and $b_1- a_1 = s_0 - t_0$. Therefore, when $b_1 - a_1 < d_{\lambda}$, then $s_0-t_0 < \la -b_1 - t_0$ and so
$s_0<\la - b_1.$  
This implies that
$$
t_0 < s_0 < \lambda - b_1 < \lambda -a_1,
$$
and so for $p$ and $q$ as in \eqref{p_k}, 
$$
t_0 = q-2p < q-p< q= s_0 < \lambda - b_1 = \la - a_1-2p < \la - a_1 - p < \lambda -a_1. 
$$
Appealing to \eqref{remind_phi} and \eqref{eq2}, we can write 
\begin{align*}
\Phi_{\lambda}(\mathbf{a}) - \Phi_{\lambda}(\mathbf{b}) 
&=\gamma(\lambda-a_1) -\gamma(\lambda-b_1)  -    \left( b_2 -a_2  \right)  \\
&=   \gamma(\lambda-a_1) - \gamma(\lambda-b_1)  -  \left( \gamma(s_0) - \gamma(t_0) \right)  \\
%&= \left[ \ga(\la - a_1) -\ga(\la - a_1 - p) + \ga(\la - a_1 - p) - \ga(\la - a_1 - 2p) \right]\\
%&- \left[ \ga(q) -\ga(q-p) + \ga(q - p) - \ga(q - 2p) \right]  \\
&= \sum_{j=0}^1 \left( \ga(\la - a_1-jp) -\ga(\la - a_1 - (j+1)p) \right) \\
&-  \sum_{j=0}^1 \left( \ga(q- jp) -\ga(q-(j+1)p) \right).  
\end{align*}
Applying the mean value theorem, there exists $h_0, h_1, h_0', h_1' \in (0,1)$ so that 
\begin{align*}
\Phi_{\lambda}(\mathbf{a}) - \Phi_{\lambda}(\mathbf{b}) 
%&= \left[ \ga'(\la - a_1-h_1p)\cdot p  + \ga'(\la - a_1 - p-h_2p)\cdot p \right]\\
%&- \left[ \ga'(q-h_3p)\cdot p  + \ga'(q - p-h_4p)\cdot p \right]  \\
&= \sum_{j=0}^1 \left( \ga'(\la - a_1-jp-h_jp)\cdot p  \right) \\
&-  \sum_{j=0}^1 \left( \ga'(q- jp-h_j' p ) \cdot p  \right),
\end{align*}
and it follows that 
\begin{equation}\label{main_case2}
\Phi_{\lambda}(\mathbf{a}) - \Phi_{\lambda}(\mathbf{b}) 
\sim \left( \sum_{j=0}^1 \left( \ga'(\la - a_1-jp-h_jp)
-  \ga'(q- jp-h_j' p ) \right) \right)  \cdot p.
\end{equation}

The purpose for adding and subtracting terms, is that the terms 
$\la - a_1-jp-h_jp$ and $q- jp-h_j' p$ are now appropriately separated for $j=0,1$.
Indeed, when $d_\lambda  >(b_1 - a_1)$,  recalling that $q=s_0$, 
it holds that 
$$(\la - a_1-jp-h_jp)-(s_0- jp-h_j' p)
=d_\la -h_jp + h_j'p
\geq d_\la/2, $$
and 
$$(\la - a_1-jp-h_jp)-(s_0- jp-h_j' p)
=d_\la -h_jp + h_j'p
\le 3d_\la/2. $$

The key point is that in \eqref{main_case2}, the arguments of $\gamma'$ within each summand are separated by a positive quantity comparable to $d_{\lambda}$. Using the bi-Lipschitz condition on $\gamma'$ (in which case $\gamma'$ is strictly monotonic on $I$), we conclude that
$$\Phi_{\lambda}(\mathbf{b}) - \Phi_{\lambda}(\mathbf{a}) \sim  d_\lambda \cdot p.$$ 
Since $p \sim (b_1 - a_1)$, this case is completed.

\smallskip
%%%%%%%%%%%%%%%%%%%%%%%%%%%%%%%%%%%%%%%%%%%%%%%%%%
%CASE 3
%%%%%%%%%%%%%%%%%%%%%%%%%%%%%%%%%%%%%%%%%%%%%%%%%%
\noindent
\textbf{Case 3:} $d_{\lambda} \le (b_1 - a_1)$. 
Set
\begin{equation}\label{pq}
p = \frac{ d_\lambda}{2} , \,\ \text{ and } \,\, 
q= (\la - b_1).
\end{equation}

With this choice of $p$ and $q$, the proof proceeds as in the  previous case. This situation is depicted below in Figure \ref{figure_case3}. 

\begin{figure}[h]
  \centering
  \includegraphics[width=12.5cm]{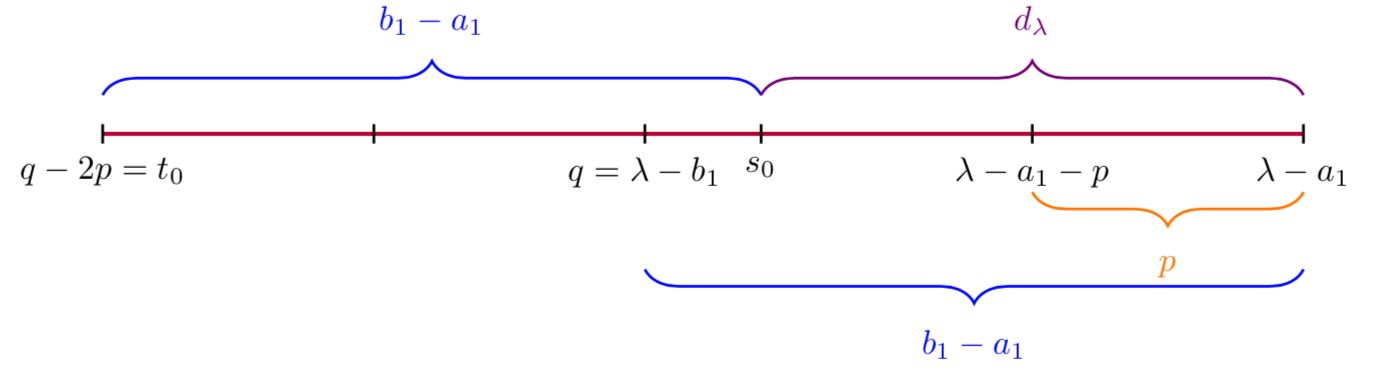}
  \caption{Case 3}
  \label{figure_case3}
\end{figure}

Using \eqref{eq1} and \eqref{distance_la_updated}, we can write  $d_\la = \la - b_1 - t_0 = \la - a_1 - s_0\geq 0$ and $b_1- a_1 = s_0 - t_0 > 0$. Therefore, when $d_\lambda \le b_1 - a_1$, then 
$ \la - b_1 - t_0 \le s_0-t_0$ and so
$\la - b_1 \le s_0$. Combining these observations, if $d_\la \le (b_1-a_1) $, then 
$$ t_0 \le  \lambda - b_1 \le s_0 \le   \lambda -a_1,
$$
and so, for $p$ and $q$ as in \eqref{pq}, 
$$t_0= q-2p \le  q-p \le q= \lambda - b_1 \le s_0 = \la - a_1 - 2p \le \la - a_1 - p \le  \lambda -a_1. 
$$

Using an identical telescoping argument as that used in the previous case to obtain \eqref{main_case2}, except now with $p$ and $q$ as in \eqref{pq}, we conclude that there exists $h_0, h_1, h_0', h_1' \in (0,1)$ so that
\begin{equation}\label{common2}
\Phi_{\lambda}(\mathbf{a}) - \Phi_{\lambda}(\mathbf{b}) 
\sim \left( \sum_{j=0}^1 \left( \ga'(\la - a_1-jp-h_jp)
-  \ga'(q- jp-h_j' p ) \right) \right)\cdot p.
\end{equation}

We now observe that 
$\la - a_1-jp-h_jp$ and $q- jp-h_j' p$ are sufficiently separated for $j=0,1$
when $d_\lambda \le b_1 - a_1$:
$$(\la - a_1-jp-h_jp)-(q- jp-h_j' p)
= b_1-a_1 -h_jp + h_j'p
\geq  (b_1-a_1)/2, $$
and 
$$(\la - a_1-jp-h_jp)-(q- jp-h_j' p)
= (b_1-a_1) -h_jp + h_j'p
\le 3(b_1-a_1)/2. $$

As in Case 2 above, we have now established the necessary separation between the arguments of $\gamma'$ in each summand; it follows that
$$\Phi_{\lambda}(\mathbf{b}) - \Phi_{\lambda}(\mathbf{b})\sim (b_1 -a_1)\cdot p.$$
Since $p \sim d_{\lambda}$, this case is finished.

%%%%%%%%%%%%%%%%%%%%%%%%%%%%%%%%%%%%%%%%%%%%%%%%%%%%%
%CURVES DON'T INTERSECT CASE
%%%%%%%%%%%%%%%%%%%%%%%%%%%%%%%%%%%%%%%%%%%%%%%%%%%%%
%
%
\smallskip

\noindent \textbf{Non-intersection case:} It remains to verify \eqref{H2} when \eqref{007} does not hold. 
Assume that $\mathbf{a}$ and $\mathbf{b}$ are such that 
\begin{equation}\label{emptycase}
  (\mathbf{a}+\Gamma)\cap(\mathbf{b}+\Gamma)= \emptyset.
\end{equation}
Let $\de>0$.  For each $\la\in A$, set
\begin{align*}
h(\la) 
&:= \Phi_{\lambda}(\mathbf{b}) - \Phi_{\lambda}(\mathbf{a}) 
=\ga(\la-b_1) - \ga(\la-a_1) + (b_2-a_2). 
\end{align*}

Relabeling if necessary, we may assume that the graph $(\mathbf{b}+\Gamma)$ is above $(\mathbf{a}+\Gamma)$ in the sense that 
for each $\la\in A$, it holds that
$$
h(\la) >0.
$$

Observe that in the case that $a_1= b_1$, then 
$h(\la) = b_2-a_2$ is constant, and so the left-hand-side of \eqref{H2} is non-zero identically when $|a-b| = |a_2 - b_2| \le \de$, in which case the right-hand-side of \eqref{H2} is bounded below by the constant $c$, and  the inequality is satisfied provided that $c$ is chosen so that $c\geq |A|$.  

Assume then that $a_1 \neq b_1$.  
We will apply a vertical shift to the curve $(\mathbf{b}+ \Gamma)$ to reduce to the intersection case considered in \eqref{007} and handled above. 
It is a consequence of the curvature assumption of Definition \ref{defn:curvature_simple} that there exists a unique $\widehat{\la}\in A$ where $h(\la)$ is minimized.  Set 
$$d:= h(\widehat{\la}).$$ (Indeed, when $a_1\neq b_1$, note that $h$ is strictly monotonic as $h'\neq 0$ by \eqref{Lips_cond}). 
Now
$$\left(  \Ga + (b_1, b_2-d) \right) \cap \left( \Ga + \mathbf{a}\right) \neq \emptyset,$$
and we see that 
\begin{equation}\label{shift}
\Phi_{\la}(\mathbf{b}) =b_2    + \gamma(\lambda -b_1)    = b_2 - d   + \gamma(\lambda -b_1)  + d 
=  \Phi_{\la}((b_1, b_2 - d )) + d.
\end{equation}

Set $ \mathbf{b}(d)=    (b_1, b_2 - d )$. 
Now, if $\la$ is such that $h(\la) =  \Phi_{\lambda}(\mathbf{b}) - \Phi_{\lambda}(\mathbf{a}) \le \de$, then 
$ \Phi_{\la}(   \mathbf{b}(d)  ) - \Phi_{\lambda}(\mathbf{a}) \le \de-d \le \de$.
Note we may assume that $\de\geq d$ since $h(\la) \geq d$ for each $\la\in A$.  
Therefore
\begin{equation}\label{containment}
\{ \la\in A:  \Phi_{\lambda}(\mathbf{b}) - \Phi_{\lambda}(\mathbf{a}) \le \de\}
\subset
\{ \la\in A:  \Phi_{\lambda}((b_1, b_2 - d )) - \Phi_{\lambda}(\mathbf{a}) \le \de\},
\end{equation}
and it follows from the the previous Cases 1-3 that there exists a constant $c>0$ that depends only on the constant $\La$ in \eqref{Lips_cond} so that 
\begin{equation}\label{bd}
|\{ \la\in A:  \Phi_{\lambda}((b_1, b_2 - d )) - \Phi_{\lambda}(\mathbf{a}) \le \de\}| \le \frac{c\,\de}{|\mathbf{b}(d)-\mathbf{a}|}.
\end{equation}

Combining \eqref{containment} and \eqref{bd}, we see that if $|\mathbf{b}(d)-\mathbf{a}|$ were bounded below by $|\mathbf{b}-\mathbf{a}|$, then the argument would be complete. Since this may not always be the case, we need a slightly more delicate analysis.  

We will now proceed in two cases, based on the relative sizes of $|b_1 - a_1|$ and $|b_2 - a_2|$. When the first difference is dominant, the shift between $\textbf{b}$ and $\textbf{a}$ is mostly horizontal and this horizontal translation is detected by the first coordinate of $\mathbf{b}(d)$. The more challenging case is when the translation is nearly vertical; this will follow the same lines as when $b_1 = a_1$. To be precise, we now consider the cases when 
$ |b_1-a_1| \geq \frac{1}{2} |b_2-a_2|$ and 
$ |b_1-a_1| < \frac{1}{2} |b_2-a_2|$ separately. 

In the former case, 
$$|b_1- a_1| \gtrsim |b-a|$$
and so 
\begin{align*}
|\mathbf{b}(d) - \mathbf{a}|^2
&= |b_1-a_1|^2  + |b_2-d-a_2|^2\\
&\geq  |b_1-a_1|^2\\
&\gtrsim |b-a|^2.
\end{align*}
In this case, we see that if $|\mathbf{b}(d)-\mathbf{a}|$ is bounded below by a constant multiple of $|\mathbf{b}-\mathbf{a}|$, and the argument is complete upon combining \eqref{containment} and \eqref{bd}.

Now consider the latter case that $ |b_1-a_1| < \frac{1}{2} |b_2-a_2|.$ Suppose that $\la$ is such that $h(\la) \le \de$.  
By the mean value theorem, there exists an $\eta$ so that 
$$ \ga(\la-b_1) - \ga(\la-a_1)=  -\ga'(\eta) (b_1-a_1).$$
Recall from \eqref{deriv_cond} that $\sup_{t\in I} |\ga'(t)| \le 1$.  
It follows from the reverse triangle inequality that
\begin{align*}
h(\la) &\geq |b_2-a_2| - |\ga'(\eta) (b_1-a_1)|\\
&\geq |b_2-a_2| - | b_1-a_1|\\
&\geq |b_2-a_2| -\frac{1}{2} |b_2-a_2|\\
&= \frac{1}{2} |b_2-a_2|\\
&\sim  |\mathbf{b}  -\mathbf{a} |,
\end{align*}
where the implicit constants are independent of $\mathbf{b}$, $\mathbf{a}$ and $\la$.  
It follows that there exists a $c'>0$ so that if  $\la$ is such that $h(\la) \le \de$, then 
$|\mathbf{b}  -\mathbf{a} | \le c'\de$ or $1 \le \frac{ c'\de  }{   |\mathbf{b}  -\mathbf{a} |}.$ Now, 
$$|\{\la\in A: h(\la) \le \de\} | \le |A| \le c \le c \frac{c' \, \de}{ |\mathbf{b}  -\mathbf{a} |},$$ provided $c$ is chosen so that $c\geq |A|$.   
\end{proof}

%%%%%%%%%%%%%%%%%%%%%%%%%%%%%%%

\subsection{Surface projections}\label{section_surfaces_trans}
%SURFACE REDUCTIONS
 Here, we show that the maps corresponding to the Favard surface length and introduced in Section \ref{subsection_Fav_higher} satisfy the transversality condition of \eqref{H2}.
We will consider the case when $\Gamma$ is a surface of revolution generated by an even, $C^2$, concave up function $f$ defined on a neighborhood of the origin. That is, $\Gamma$ will be the graph of $\gamma : \mathbb{R}^2 \to \mathbb{R}$ given by
$$\gamma(s) = f(|s|)$$
defined on a closed ball $B:= \overline{B(0, L)}$ for some $L > 0$. 

Note that $f''>0$ on $[-L, L]$.  
It is straightforward to check that $f'(x) \geq 0$ on $[0, L]$ with equality only at $x=0$; computing the second partial derivative of $\ga$ in $x$ at $|(x,y)|=0$ and $|(x,y)|\neq 0$ separately shows that there exists $c>0$ so that for each $(x,y) \in B$ 
\begin{equation}\label{deriv_nonzero}
\frac{\partial^2\gamma}{\partial x^2}(x,y) >c.
    \end{equation}

Now, we choose a parameter set $A$ and a domain $\Om$ as in Definition \ref{assumption:big}: 
set  
$A =  \overline{B(0,\frac{L}{3})} \subset \R^2$ and $\Om= \overline{B(0,\frac{L}{3}) }\subset \R^3$. 
For $\al \in A$, 
denote the vertical line
 $$\ell _\al:=\left\{(x,y, z):(x,y)=\al\right\}.$$ 
If $\al= (\al_1, \al_2) \in A$ and $\mathbf{a}=(a_1, a_2, a_3)\in \Om$, note $(\alpha_1 - a_1, \alpha_2 - a_2) \in B$ and
 $$\ell _\al\cap(\mathbf{a}+\Gamma) = (\al_1, \, \al_2, \, a_3 +  \ga(\al_1 - a_1, \al_2 - a_2))$$ is a singleton. 
  Thus, we can 
define the two-parameter family of mappings $\{\Phi_\al(\mathbf{ \cdot })\}_{\al\in A}$,
$\Phi_\al:\Om\to \R$
by
 \begin{equation}\label{remind_phi_3d}
\Phi_\al(\mathbf{a}) = a_3 + \ga(\al_1 - a_1, \al_2 - a_2).
 \end{equation}

The following lemma states that this family of maps satisfy the transversality condition of \eqref{H2} when $\Ga$ is a surface of revolution of this form. 

\begin{lem}\label{lem_surface_maps_trans}(Surface maps are transversal). 
Let $\Ga= \{(t, \ga(t)): t\in B\}= \{(t, f(|t|)): t\in B\}$ be a surface of revolution with $f:\R\rightarrow \R$,  $\ga:\R^2 \rightarrow \R$ as defined above so that \eqref{deriv_nonzero} holds on $B=\overline{B(0,L)}$. 
With the notation above, the associated family of projections $\{\Phi_\al: \Om \rightarrow \R: \al \in A\}$ is $1$-transversal in the sense of Definition \ref{definition:s-transversal}. 
\end{lem}

 While the proof of Lemma \ref{lem_surface_maps_trans} is similar to its $2$-dimensional analogue, Lemma \ref{curve_trans}, there is a new layer of complexity that arises.  
  In the $2$-d case, in which $\Ga$ was a curve and the graph of a real valued function, the intersection set $\left( \mathbf{a} + \Ga \right) \cap \left( \mathbf{b} + \Ga \right)$ consisted of at most one point.  
Denoting this point by  $\mathbf{x}=(x_1,x_2)$ (when it exists) and setting 
 $H(\la):= |\Phi_\la(\mathbf{a}) - \Phi_\la(\mathbf{b}) |$, 
 with $\Phi_\al$ as in \eqref{remind_phi},
we saw that
 $H(x_1)=0$ and observed that $H$ grows at a linear rate in a neighborhood of $x_1$. 
 In the $3$-d case, in which $\Ga$ is a surface, the set $\left( \mathbf{a} + \Ga \right) \cap \left( \mathbf{b} + \Ga \right)$ may consists of many points. Here, we show that the the function $H(\la)$, now with $\Phi_\al$ as in \eqref{remind_phi_3d}, obeys a similar linear growth condition along horizontal lines.  We now prove Lemma \ref{lem_surface_maps_trans} using the set-up above, 
and we begin with a few simplifying reductions. 
 
\begin{proof}
By rescaling in the $z$-axis, we may assume that all the first partial derivatives of $\gamma$ are bounded by $1$. For distinct $\mathbf{a}, \mathbf{b}\in \Om$, our aim is to verify that 
$$|\{\la \in A :| \Phi_\la(\mathbf{a}) - \Phi_\la(\mathbf{b}) |\le \de \}| \lesssim \frac{\de}{|\mathbf{a}-\mathbf{b}|}.$$
Translating, it is enough to consider the situation when $\mathbf{a} = (0,0,0)$. Further, since $\Ga$ is symmetric about the origin, it suffices to consider the case when $\mathbf{b} = (b_1, 0, b_3)$ for $b_1, b_3\geq 0$. 

After this reduction, our goal is to show that
\begin{equation}\label{mama3d}
  |\{  \lambda \in A: |\Phi_\lambda(\vec{0}) - \Phi_\lambda(\mathbf{b}) | \le \de \}|  
  \lesssim \frac{\de}{|\mathbf{b}|},
  \end{equation}
  for a universal constant independent of $\mathbf{b}$ and $\de$. To this end, fix the coordinate $\lambda_2$ and form a slice parallel to the $xz$-plane; we will show that 
  \begin{equation}\label{littlemama3d}
  |\{  \lambda_1: \la= (\la_1, \la_2) \in A \text{ and }  |\Phi_\lambda(\vec{0}) - \Phi_\lambda(\mathbf{b}) | \le \de \}|  \lesssim \frac{\de}{|\mathbf{b}|}
  \end{equation}
  for a universal constant independent of $\la_2$. Once this is completed, we may integrate the estimate with respect to $\lambda_2$ over the interval $[-\frac{L}{3}, \frac{L}{3}]$ and apply Fubini's theorem to recover \eqref{mama3d}.  
  Note that $|\cdot|$ in \eqref{mama3d} denotes the $2$-dimensional Lebesgue measure and $|\cdot|$ in \eqref{littlemama3d} denotes the $1$-dimensional Lebesgue measure.
  
We are now working within a two-dimensional slice of the surface and will be able to apply the results of Section \ref{Favard curve section}. Note that the slice
$$\gamma_{\lambda_2} := \Ga\cap \{y=\la_2\} = \{ (t_1, \la_2, \ga(t_1,\la_2)): (t_1, \la_2) \in B\}$$
forms a curve in the plane $\{y = \lambda_2\}.$ Since $\mathbf{b} = (b_1, 0, b_3)$, the translated surface $(\Gamma + \mathbf{b})$ also intersects this plane in a curve
$$(\Ga + \mathbf{b}) \cap \{y = \la_2\} = \{(s_1 + b_1, \la_2, \ga(s_1, \la_2) + b_3) : (s_1, \la_2) \in B\}.$$
The key point is that this curve is merely a translate of $\gamma_{\lambda_2}$:
$$(\Gamma + \mathbf{b}) \cap \{y = \lambda_2\} = \gamma_{\la_2} + \mathbf{b}.$$

Recalling the curvature condition \eqref{deriv_nonzero}, we see that the curve $\gamma_{\la_2}$ satisfies the simple curvature condition of Definition \ref{defn:curvature_simple}. Applying Lemma \ref{curve_trans} (in particular, the result of \eqref{1-trans}) then establishes \eqref{littlemama3d} as desired.
\end{proof}

%%%%%%%%%%%%%%%%%%%%%%%%%%%%%%%%%%%%%%%%%%%%%%%%%%%%%%
%ENERGY SECTION
%%%%%%%%%%%%%%%%%%%%%%%%%%%%%%
\section{Energy techniques for pushforwards}\label{energy_section}
We now turn to measure estimates using the energy and potential based approach of Mattila \cite{Mat90}. The key idea here will be that the energies associated to a measure $\mu$ and its pushforwards $\widetilde{\pi_{\alpha}}_{\sharp} \mu$ are closely related. This will allow us to prove strong asymptotic lower bounds for the Favard curve lengths of neighborhoods of sets. First, we begin by proving Proposition \ref{thm:basic_marstrand}, illustrating how transversality plays a role in the study of pushforward measures. This proposition provides a generalization of Marstrand's result on the typical dimension of projections to a nonlinear setting, and the proof provided here is similar to that which appeared in Solomyak's \cite{Sol98} in the context of general metric spaces.

Recall that we have a family $\{\proj : \alpha \in A\}$ of maps into an $m$-dimensional space, $\dimh E = t \le m$, and the family of projections is $m$-transversal. Our goal is to show that for $\psi$-almost every $\alpha \in A$,
$$\dimh \proj E = t.$$
The primary tool will be to use that if $x$ and $y$ are two fixed points, then the projection operators $\proj$ will usually be able to distinguish between $x$ and $y$ on scale $|x - y|$. This is quantified with the distribution function.

\begin{proof}[Proof of Proposition \ref{thm:basic_marstrand}]
Suppose that $E$ supports a Borel probability measure $\mu$ with finite $\tau$-energy. 
Recall the energy of the measure $\mu$, $I_\tau(\mu)$, is defined in \eqref{energy} and the pushforward,  $\proj{}_{\sharp} \mu$,  is defined in \eqref{pushforward}. Averaging over the set of parameters and computing the energies of the pushforward measures, we have
\begin{align*}
\int_A I_{\tau}(\proj{}_{\sharp} \mu) \, d\psi(\alpha) &= \int_A \iint \frac{1}{|u - v|^{\tau}} \, d\proj{}_{\sharp} \mu(u) \, d\proj{}_{\sharp} \mu(v) \, d\psi(\alpha) \\
&= \int_A \iint \frac{1}{|\proj(x) - \proj(y)|^{\tau}} \, d\mu(x) \, d\mu(y) \,d\psi(\alpha) \\
&= \iint \int_A \frac{1}{|\proj(x) - \proj(y)|^{\tau}} \, d\psi(\alpha) \, d\mu(x) \, d\mu(y) \\
&= \iint \left[\int_A \frac{|x - y|^{\tau}}{|\proj(x) - \proj(y)|^{\tau}} \, d\psi(\alpha)\right] \frac{d\mu(x) \, d\mu(y)}{|x - y|^{\tau}}.
\end{align*}
We can study the innermost integral using the transversality condition together with the distribution function:
\begin{align*}
\int_A \frac{|x - y|^{\tau}}{|\proj(x) - \proj(y)|^{\tau}} \, d\psi(\alpha) 
&= \int_0^{\infty} \psi \left(\left\{\alpha : \frac{|x - y|^{\tau}}{|\proj(x) - \proj(y)|^{\tau}} \ge r\right\}\right) \, dr \\ 
&= \int_0^{\infty} \psi\left(\left\{\alpha : |\proj(x) - \proj(y)| \le r^{-1/{\tau}} |x - y|\right\}\right) \, dr \\
&= \tau \int_0^{\infty} \psi\left(\left\{ \alpha : |\proj(x) - \proj(y)| \le \de |x - y|\right\} \right) \, \frac{d\de}{\de^{1 + {\tau}}}.
\end{align*}
For a fixed $\de_0 > 0$, the integral on $[\de_0, \infty)$ converges: our parameter set has finite measure, and $\int_{\de_0}^{\infty} \frac{d\de}{\de^{1 + {\tau}}}$ is finite. Therefore, we only need to consider the case of $\de \in [0, \de_0)$; this corresponds to the set of parameters which are not able to distinguish $x$ and $y$, and will have small measure due to transversality. In particular, the $m$-transversality of \eqref{H2alt} with $s=m$ yields
$$\psi \left(|\left\{ \alpha : |\proj(x) - \proj(y) \le \delta |x - y|\right\} \right) \lesssim \de^m,$$
for all $\delta \le \delta_0$, implying that
\begin{align*}
\int_0^{\de_0} 
\psi \left(\left\{ \alpha : |\proj(x) - \proj(y)| \le \de |x - y|\right\} \right) \frac{d\de}{\de^{1 + {\tau}}} &\lesssim  \int_0^{\de_0} \de^{m - {\tau}} \, \frac{d\de}{\de}.
\end{align*}
This converges provided that $\tau < m$. 
We have now shown that for $\tau<m$,
$$\int_A I_{\tau}(\proj{}_{\sharp} \mu) \, d\alpha \lesssim
\iint  \frac{d\mu(x) \, d\mu(y)}{|x - y|^{\tau}} = I_{\tau}(\mu) < \infty,$$
and therefore the energy $I_\tau(\proj{}_{\sharp} \mu)$ is finite for $\psi$-almost every $\alpha$.

To finish the proof, recall that if $E$ has positive $\mathcal{H}^t$ measure, then for any $\tau < t$ there exists a measure $\mu$ supported on $E$ with finite $\tau$-energy (see Frostman's lemma in \cite{Mat15}). It follows that if  $\tau< t\le m$, then
the pushforward $\proj{}_{\sharp} \mu$ will also have finite $\tau$-energy, implying that $\proj(E)$ has Hausdorff dimension at least $\tau$. Passing to a countable sequence $\tau_n$ converging upwards to $t$ gives the desired result.
\end{proof}

For the remainder of the section, we will employ the notation of Definition \ref{assumption:big}. 
Recall that the lower derivative of the measure $\proj_{\sharp} \mu$ with respect to $h$ at the point $u$ is defined by 
$$\underline{D}(\proj_{\sharp} \mu, h, u) = \liminf_{\delta \to 0} \frac{\proj_{\sharp} \mu(B(u, \delta))}{h(B(u, \delta))}.$$
The upper derivative is similarly defined, taking the limit supremum. In the case that the lower and upper derivatives coincide, they will agree with the Radon-Nikodym derivative denoted $D(\proj_{\sharp} \mu, h, u)$. 

\begin{lem}[Absolute continuity of pushforwards]
\label{theorem:l2_integrability}
Suppose that $\{\proj : \alpha \in A\}$ is an $s$-transversal family of maps and that $\psi$ is a Borel measure on $A$. If $\mu$ is a Borel measure with compact support contained in $\Omega$ and $I_s(\mu) < \infty$, then for $\psi$-almost every $\alpha$, we have that $\proj_{\sharp} \mu \ll h$ and
$$\int_A \int_X D(\proj_{\sharp} \mu, h, u)^2 \, dh(u) d\psi(\alpha) \lesssim I_s(\mu).$$
\end{lem}

\begin{proof}
Consider the integral
$$\iint \underline{D}(\proj_{\sharp} \mu, h, u) d\proj_{\sharp}\mu(u) d\psi(\alpha).$$
Due to the joint continuity assumption for the functions $(x, \alpha) \mapsto \proj(x)$, the integrands will be measurable with respect to the appropriate measures (each of which are Borel measures).  We now follow the definition of the lower derivative along with Mattila's approach:
\begin{align}
\iint \underline{D}(\proj_{\sharp} \mu, h, u) &d\proj_{\sharp}\mu(u) d\psi(\alpha) \nonumber \\
&= \iint \liminf_{\delta \to 0} \frac{\proj_{\sharp}\mu\big(B(u, \delta)\big)}{h(B(u, \delta))}\,  d\proj_{\sharp}\mu(u) d\psi(\alpha) \nonumber \\
&\lesssim \liminf_{\delta \to 0} \frac{1}{\delta^m} \iint \proj_{\sharp} \mu \big(B(u, \delta)\big) \, d\proj_{\sharp}\mu(u) d\psi(\alpha) \nonumber \\
&= \liminf_{\delta \to 0} \frac{1}{\delta^m} \iint \mu\big(\proj^{-1} B(u, \delta)\big) \, d\proj_{\sharp}\mu(u) d\psi(\alpha) \nonumber \\
&= \liminf_{\delta \to 0} \frac{1}{\delta^m} \iint \mu \big\{y : \proj y \in B(u, \delta)\big\} \, d\proj_{\sharp}\mu(u) d\psi(\alpha). \label{equation:fatou}
\end{align}
Pushforward measures obey the identity
$$\int g \, df_{\sharp}\nu = \int (g \circ f) \, d\nu$$
for non-negative Borel functions $f$ and $g$ and a Borel measure $\nu$. Applying this to the function $g(u) := \mu \left\{y : \proj y \in B(u, \delta)\right\}$, we find that
\begin{align}
\iint \mu \big\{y : \proj y \in B(u, \delta)\big\} \, &d\proj_{\sharp}\mu(u) d\psi(\alpha) \nonumber \\
&= \iint \mu \left\{y : \proj y \in B(\proj x, \delta)\right\} \, d\mu(x) d\psi(\alpha) \nonumber \\
&= \iint \psi \left\{\alpha : \proj y \in B(\proj x, \delta)\right\} \, d\mu(x) d\mu(y) \nonumber \\
&= \iint \psi\left(\left\{\alpha : \operatorname{dist}(\proj x, \proj y) \le \delta\right\}\right) \, d\mu(x) d\mu(y). \label{equation:distance}
\end{align}
Combining equations \eqref{equation:fatou} and \eqref{equation:distance}, we get
\begin{align}\label{equation:energy_precursor}
\iint \underline{D}(\proj_{\sharp} \mu, h, u) &d\proj_{\sharp}\mu(u) d\alpha \\ \nonumber
&\lesssim \liminf_{\delta \to 0} \iint \frac{\psi(\left\{\alpha : |\proj x - \proj y| \le \delta\right\})}{\delta^m} \, d\mu(x) d\mu(y).
\end{align}
We are now ready to apply the $s$-transversality condition. Since
$$\psi \left( \{\alpha : |\proj(x) - \proj(y)| \le \delta\}\right) \lesssim \frac{\delta^m}{|x - y|^s},$$
we find that
\begin{equation} \label{equation:energy}
\iint \underline{D}(\proj_{\sharp} \mu, h, u) d\proj_{\sharp}\mu(u) d\psi(\alpha) 
\lesssim \iint \frac{d\mu(x) \, d\mu(y)}{|x - y|^s}
= I_s(\mu).
\end{equation}

Since $I_s(\mu)< \infty$, we conclude that for $\psi$-almost every $\alpha$, the lower derivative $\underline{D}(\proj_{\sharp} \mu, h, u)$ is finite for $\proj_{\sharp}\mu$ almost every $u \in X$. Following \cite[Thm 2.12]{Mat95}, this implies that $\proj_{\sharp} \mu \ll h$ for all such parameters, in which case $D(\proj_{\sharp} \mu, h, u)$ exists for $\proj_{\sharp} \mu$-a.e. point $u \in X$. Finally, we can use Fubini's theorem to conclude that
\begin{equation}\label{equation:fubini}
\int_X D(\proj_{\sharp} \mu, h, u)^2 \, dh(u) = \int_{X} D(\proj_{\sharp} \mu, h, u) \, d\proj_{\sharp} \mu(u).
\end{equation}
The combination of \eqref{equation:fubini} with the estimate \eqref{equation:energy} establishes the desired result.
\end{proof}

The next result follows from Lemma \ref{theorem:l2_integrability} and, in essence, states that nonlinear variants of Favard length are controlled from below by the energy of any nice measure placed on the set.

\begin{lem}[Lower bound on average projection length]\label{theorem:general_favard}
Suppose that $\{\proj : \alpha \in A\}$ is an $s$-transversal family of maps with a Borel probability measure $\psi$ on $A$. If $\mu$ is a Borel probability measure supported on a compact set $F \subseteq \Omega$, then
$$\int_A \left(h(\proj F)\right)^{-1} d\psi(\alpha) \lesssim I_s(\mu)$$
and
\begin{equation}\label{main_lower}
\frac{1}{I_s(\mu)} \lesssim \int_A h(\proj F) \, d\psi(\alpha).
\end{equation}
\end{lem}

This is an analogue of of \cite[Theorem 3.2]{Mat90}.  The proof relies on Lemma \ref{theorem:l2_integrability}.
\begin{proof}
 Since $F \subseteq \proj^{-1}(\proj F)$ and $\mu$ is a probability measure, we can apply the definition of the pushforward to conclude that
$$1 = \proj_{\sharp} \mu \left(\proj F\right)^2 = \left(\int_{\proj F} D(\proj_{\sharp} \mu, h, u) \, dh(u)\right)^2.$$
Invoking the Cauchy-Schwartz inequality,
\begin{align*}
1 &\le h(\proj F) \int_{\proj F} D(\proj_{\sharp} \mu, h, u)^2 \, dh(u)
\end{align*}
for all $\alpha \in A$. After dividing both sides by $h(\proj F)$, integrating in $\psi$, and invoking Lemma \ref{theorem:l2_integrability}, we have
\begin{align*}
\int_A \left(h(\proj F)\right)^{-1} \, d\psi(\alpha) &\le \int_{A} \int_{\proj F} D(\proj_{\sharp} \mu, h, u)^2 \, dh(u) \, d\psi(\alpha) \\
&\le \int_{A} \int_{X} D(\proj_{\sharp} \mu, h, u)^2 \, dh(u) \, d\psi(\alpha) \\
&\lesssim I_s(\mu),
\end{align*}
thus establishing the first inequality.

For the second part of the theorem, consider the function $f(\alpha) = h(\proj F).$ Applying the Cauchy-Schwarz inequality to $1= \int d\psi= \int f^{1/2} \cdot f^{-1/2} d\psi$ immediately gives the claimed result.
\end{proof}

In order to apply Lemma \ref{theorem:general_favard} to neighborhoods, we construct a measure with appropriate support and obtain an upper bound on the energy. The following lemma says that whenever the dimension of $F$ is known, there is at least one auxiliary measure supported on the neighborhood $F(r)$ whose energy is easily computable. This is the final tool that we will need in order to estimate the average projection size of a neighborhood, and it comes directly from  \cite[Theorem 4.1]{Mat90}. We give a summary of the main idea of the construction.
\begin{lem}[Construction of auxiliary measure]\label{lem_energy_bound}
Let $0<s\le m$.  
Suppose $\mu$ is a Borel probability measure supported on a compact set $F\subset \R^n$ and there exists $c>0$ so that 
$$\mu(B(x,r)) \le cr^s$$
for each $x\in \R^n$ and $r>0$.  
Then for each $r\in (0,1)$, there exists a probability measure $\nu$ supported in $F(2r)$ so that
\begin{equation}\label{energy_bound_1}
I_s(\nu) \lesssim r^{s-m} \,\,\, \text{ if  $s<m$}
\end{equation}
\begin{equation}\label{energy_bound_2}
I_s(\nu) \lesssim \log{\left(\frac{1}{r}\right)} \,\,\, \text{ if  $s=m$.}
\end{equation}
\end{lem}

\begin{proof}[Summary of proof]
For $F$, $\mu$ and $r$ as in the statement of Lemma \ref{lem_energy_bound}, 
we can use a covering argument to find a disjoint collection of balls $\{B_i\}_{i = 1}^k$, each with radius $r$, so that 
$\tau := \mu \left(\bigsqcup_{i = 1}^k B_i\right) >0.$
The measure $\nu$ is then defined to be
  \begin{equation}\label{general_nu}  
  \nu(A) := \frac{1}{\tau} \sum_{i = 1}^k \mu(B_i) \frac{|A \cap B_i|}{ |B_i|},
    \end{equation}
where $|\cdot|$ denotes the $n$-dimensional Lebesgue measure.  
Note that $\nu$ is supported in $F(2r)$. The $\nu$ measure of a ball of radius $u$ can be bounded from above, considering the cases when $u\le r$, $r \le u \le 1$, and $1\le u$ separately, and a computation with the distribution function (analogous to the computations in the proof of Proposition \ref{thm:basic_marstrand}) shows that  $ I_m(\nu) \lesssim r^{s - m}$  when $s<m$.  Further, when $s=m$, a similar computation shows that 
$ I_m(\nu) \lesssim \log{\left(\frac{1}{r}\right)}.$
\end{proof}

With Lemmas \ref{theorem:general_favard} and \ref{lem_energy_bound} in tow, we now turn to the proof of Theorem \ref{main}. In this context, it will be important that the family is $m$-transversal, with $m$ matching the dimension of the target set. 

%PROOF OF THEOREM 
\begin{proof}[Proof of Theorem \ref{main}]
Assume that $\{\proj : \alpha \in A\}$ is $m$-transversal. 
Letting $F$ and $\mu$ be as in the hypotheses, we can use Lemma \ref{lem_energy_bound} to construct the auxiliary measure $\nu$ with computable energy. Applying the estimate \eqref{main_lower} of Lemma \ref{theorem:general_favard} to $\nu$ and $F(2r)$ yields the theorem.
\end{proof}

%%%%%%%%%%%%%%%%%%%% 
%%%% SECTION 
\section{Applications and examples}\label{section_apps}
\subsection{Proving Theorems \ref{theorem:smooth_surface_visibility} and \ref{FavTheorem}}\label{section:main_proofs}

We now turn to self-contained proofs of the applications to Favard curve length and visibility, respectively.

\begin{proof}[Proof of Theorem \ref{FavTheorem}]
In the case that $\Ga$ satisfies the simple curvature assumption of Definition \ref{defn:curvature_simple}, we can apply Lemma \ref{curve_trans} to conclude that the curve projections associated to $\FavG$ form a $1$-transversal family, and the theorem follows from
Theorem \ref{main}.
The reductions made at the beginning of Section \ref{Favard curve section} imply that establishing the theorem for this special class of $\Ga$ suffices. 
\end{proof}

On the other hand, the visibility result requires a little bit more analysis, since transversality depends on the relative geometry of the visibile set and the vantage set. 

\begin{proof}[Proof of Theorem \ref{theorem:smooth_surface_visibility}]
In Lemma \ref{lemma:tube_transversal}, we established that the family of radial projections $\{P_a : a \in A\}$ is $(n - 1)$-transversal provided that the underlying probability measure $\psi$ supported on $A$ satisfies the tube condition with respect to $E$:
$$\psi(T_{\delta}) \lesssim \delta^{n - 1}.$$

In our context, $\psi = \mathcal{H}^{n - 1}$ and it suffices to show that there exists a positive $\delta > 0$ such that for any tube $T_{\delta}$ passing through the visible set $E$, we have
$$\mathcal{H}^{n - 1}(T_{\delta} \cap A) \lesssim \delta^{n-1}.$$
However, this follows immediately from the tangent plane condition: the angle between the tube $T_{\delta}$ and any tangent plane to $A$ is uniformly bounded away from zero and the claim follows. Now that transversality has been established, we conclude the proof with an application of Theorem \ref{main} as in the previous argument. 
\end{proof}

A slightly more general version of Theorem \ref{theorem:smooth_surface_visibility} is available without separation between the vantage set $A$ and the visible set $E$.  
The tube condition is also guaranteed 
 upon replacing our tangent plane condition with the following slightly more technical statement:
there exists $\de_0>0$ and 
$\theta_0 \in (0,\frac{\pi}{2})$ so that,
 %upon replacing our tangent plane condition with the following slightly more technical statement:
%Recall, for $x\neq y$, $L_{x,y}$ denotes the line through $x$ and $y$ and 
%$L_{x,y}(\de)$ denotes its $\de$-thickening. $A_a$ denotes the tangent line to $A$ at $a$.  
%Suppose $A$ satisfies the following: 
if $a\in L_{x,y}(\de_0)\cap A$ for distinct $x,y\in E$, then $A_a$ meets $L_{x,y}$ at an angle of at least $\theta_0$. 

\subsection{Applications to dynamically generated sets}\label{section:application_dynamic}
A key tool in proving Theorem \ref{main} was to establish the existence of an auxiliary measure $\nu$ supported near $F$ whose $s$-energy is easily computable. Lemma  \ref{theorem:general_favard} then relates the average projection length to the energy. In the case of many fractal sets, we can construct the special measure $\nu$ in a geometrically motivated \textit{ad hoc} manner. We now turn to the proof of Corollary \ref{corollary:four_corner_favard}.

\begin{proof}[Proof of Corollary \ref{corollary:four_corner_favard}]
Set $s=m=1$ and $r= \left(\frac{1}{4}\right)^n$. 
Recall $\K_n$ denotes the $n$-th generation in the construction of the four-corner Cantor set $\K$. We can write $\K_n$ as the union of $4^n$ squares $Q_i$ of side length $4^{-n}$, and 
define a probability measure on $\nu$ supported on $\K_n$ by 
$$\nu(A) =  \sum_{i=1}^{4^n}  \frac{ |A\cap Q_i|}{ \left(\frac{1}{4}\right)^n}.$$
This is the equidistributed measure on $\K_n$ (and can be compared to the constructed measure of Lemma \ref{lem_energy_bound} when $\mu$ denotes the $1$-Hausdorff measure restricted to $\mathcal{K}$). 

Observe $\nu(\K_n) =1$ and 
\[  \nu(B(x, u)) \sim  \left\{  
\begin{array}{ll}
      u^2/r  & \text{   for }   u\le r \\
      u        & \text{   for }        u\geq r \\
     1         & \text{   for }    u\geq 1. \\
\end{array} 
\right. \]
A direct estimate of the energy integral leads to
\begin{equation}\label{butter}
I_1(\nu) \sim \log{\left( \frac 1 r \right)} \sim n.
\end{equation}

Next, as we  have already established in Lemma \ref{curve_trans} that the curve projections which lead to $\FavG$ are a $1$-transversal family under our simple curvature assumption, we can apply Lemma \ref{theorem:general_favard} to conclude that
\begin{equation}\label{sugar}
\frac{1}{ I_1(\nu)} \lesssim \int_\R |\Phi_\al(\K_n)| d\al.
\end{equation}

Combining \eqref{butter} and \eqref{sugar} completes the proof of Corollary \ref{corollary:four_corner_favard}.
\end{proof}
%%%%%%

%%%%%%%%%%%%%%%%%%%%%%%%%%%%%%%%%%%%%%%%%%%%
It is worth emphasizing that the main point here is the existence of the measure $\nu$ with easily bounded energy at the appropriate dimension. As such, these techniques apply to a much broader class of fractal sets at dimension $1$; whenever we can have a piece-counting argument that gives a sharp estimate for $I_1(\nu)$, we will get a similar bound. This is frequently the case for fractals that are generated by an iterated function system, including $\K_n$. 

Next, we give the corresponding applications for visibility:

\begin{proof}[Proof of Corollary \ref{corollary:four_corner_visibility}]
Since no tangent line to the curve $\Gamma$ passes through the compact set $[0, 1]^2$, there is a positive distance between any tangent line to $\Gamma$ and $\K_n$. This is the two-dimensional version of the non-tangency assumption of Theorem \ref{theorem:smooth_surface_visibility} and thus the family of radial projections $\{P_a : a \in \Gamma\}$ is $1$-transversal. Again taking $\nu$ to be the equidistributed measure on $\K_n$, the corollary now follows from Lemma \ref{theorem:general_favard} and the estimate \eqref{butter}. 
\end{proof}

\subsection{Projections without transversality}\label{section:nonexamples}
In each of the cases handled above, a notion of transversality 
is used 
to show that the set of parameters which cannot distinguish two nearby points on an appropriate scale is rather small. 
One may ask whether such a condition is necessary. 
In the following examples, we explore what can happen when transversality is absent.  

\begin{example}[Asymptotic $\Fav_\Ga$ that decays too fast]
Suppose that the curve $\Gamma$ is $x$-axis in $\mathbb{R}^2$, suppose $F$ is a horizontal line segment, and consider the curve projections $\Phi_\al$ of Section \ref{Favard curve section}.  
Then Theorem \ref{Favard curve section} fails.  
\end{example}
Recalling \eqref{eq_formulations}, we see that
$$\operatorname{Fav}_{\Gamma}(F(r)) \sim 2r.$$
This tends to zero much more rapidly than $(\log r^{-1})^{-1}$.

Our next example illustrates that Favard curve length does not necessarily detect rectifiability without a transversality assumption.  In particular, without a curvature assumption, it is possible to have a purely unrectifiable set with positive and finite Hausdorff $1$-measure, which has strictly positive Favard curve length. 

\begin{example}[A lower bound that does not decay]
Suppose that $\Gamma$ is a straight line in $\mathbb{R}^2$ passing through the origin with slope $\frac 1 2$ (or angle $\theta = \arctan \frac 1 2$) and that $F$ is the $4$-corner Cantor set. 
Consider the curve projections $\Phi_\al$ of section \ref{Favard curve section}. Then for all $\alpha$ so that $\Phi_\al$ is defined on $\K$, the projection $\Phi_\al \K$ is an interval with length comparable to $1$.
\end{example}

 To see this, consider the first generation of the four corner Cantor set $\mathcal{K}$ and its four constituent squares. Each square projects to an interval. Since the line has slope $1/2$, the points $(1/4, 0)$ and $(3/4, 1/4)$ project to the same position within $L_{\alpha}$. Similarly, $(0, 1/4)$ and $(1, 3/4)$ share a projection and so do $(1/4, 3/4)$ and $(3/4, 1)$.
Therefore, the projection of the bottom right square is a segment connecting $\proj (1, 0)$ and $\proj (3/4, 1/4)$; the projection of the lower left square is a segment connecting $\proj (1/4, 0)$ and $\proj (0, 1/4)$, and so on. The four intervals found in this manner only meet at their endpoints, and their union is an interval with length greater than $1$.  
Finally, an application of self-similarity shows that this argument works for the second generation of the Cantor set as well; this extends to all subsequent generations and $\mathcal{K}$ itself. 

As a final example, we see
what happens for visibility when we do not assume the tube condition. 

\begin{example}[co-planar sets lack the tube condition]\label{notube}
Suppose $A$ and $E$ are as in Theorem \ref{theorem:smooth_surface_visibility} so that 
 $A$ is a smooth $(n-1)$-dimensional surface, $E$ has positive $s$-dimensional Hausdorff measure, and 
 $|a-e|\lesssim 1$ for each $a\in A$ and $e\in E$.
Moreover, assume that $A$ and $E$ are both subsets of the same hyperplane in $\R^n$. 
Consider the radial projections $P_a$ of Section \ref{radial}.  
Then the lower bounds of Theorem \ref{theorem:smooth_surface_visibility} fail when $s>n-2$.  
\end{example}

For $A$ and $E$ in $\R^n$ and $a\in A$, the radial projection $P_a(E)$ is a set of Hausdorff dimension at most $n-1$.  Embedding $A$ and $E$ in the same hyperplane guarantees that $P_a(E)$ is a set of Hausdorff dimension at most $n-2$. As such, it can be covered by $C \left(\frac{1}{r}\right)^{n-2}$ balls of radius $r$, for some $C$.  Since the $(n-1)$-dimensional measure of a ball is of order $r^{n-1}$, we conclude that the $(n-1)$-dimensional Hausdorff measure restricted to $S^{n-1}$ is bounded by $|P_a(E(r))| \lesssim r$. Since $r \ll log(\frac 1 r)^{-1} $ and $r \ll r^{n-1-s}$ whenever $n-2<s$ and $r$ is sufficiently small, both the first and second estimates of Theorem \ref{theorem:smooth_surface_visibility} fail in this regime.  

To see what goes awry in Example \ref{notube}, note that the tube $L_{x,y}(\delta)$ for distinct $x,y\in E(r)$ intersects $A$ in a set of measure $\de^{n-2} \gg \de^{n-1}$ and the upper bound required by the tube condition in \eqref{equation:tube} fails. In this case, $P_a(E)$ for $a\in A$ fails to differentiate the points of $E$.

As an explicit example of what fails, consider the case that $n = 2$. When $A$ and $E$ are contained in the same line, then $P_a(E)$ consists of at most two points for any $a \in A$. This means that the projections $P_a$ cannot differentiate points in $E$.

\bibliography{refs}
\bibliographystyle{abbrv}

\end{document}